\documentclass[11pt, a4paper]{amsart}
\usepackage{amsmath}
\usepackage{amssymb}
\usepackage{pdfsync}
\usepackage{hyperref}
\usepackage{color}

\newtheorem{prop}{Proposition}[section]
\newtheorem{teo}{Theorem}[section]
\newtheorem{lema}{Lemma}[section]
\newtheorem{coro}{Corollary}[section]

\theoremstyle{definition}

\def\ep{\varepsilon}

\def\a{\alpha}

\def\R{\mathbb R}

\renewcommand{\theequation}{\arabic{section}.\arabic{equation}}

\begin{document}
\title[Asymptotics for time-fractional diffusion]{Asymptotic profiles for inhomogeneous heat equations with memory}

\author[C. Cort\'{a}zar,  F. Quir\'{o}s \and N. Wolanski]{Carmen Cort\'{a}zar,  Fernando  Quir\'{o}s, \and Noem\'{\i} Wolanski}

\address{Carmen Cort\'{a}zar\hfill\break\indent
Departamento  de Matem\'{a}tica, Pontificia Universidad Cat\'{o}lica
de Chile \hfill\break\indent Santiago, Chile.} \email{{\tt
ccortaza@mat.puc.cl} }

\address{Fernando Quir\'{o}s\hfill\break\indent
Departamento  de Matem\'{a}ticas, Universidad Aut\'{o}noma de Madrid,
\hfill\break\indent 28049-Madrid, Spain,
\hfill\break\indent and Instituto de Ciencias Matem\'aticas ICMAT (CSIC-UAM-UCM-UC3M),
\hfill\break\indent 28049-Madrid, Spain.} \email{{\tt
fernando.quiros@uam.es} }

\address{Noem\'{\i} Wolanski \hfill\break\indent
IMAS-UBA-CONICET, \hfill\break\indent Ciudad Universitaria, Pab. I,\hfill\break\indent
(1428) Buenos Aires, Argentina.} \email{{\tt wolanski@dm.uba.ar} }



\date{}

\begin{abstract}
We study the large-time behavior in all $L^p$ norms of solutions to an inhomogeneous nonlocal heat equation in $\mathbb{R}^N$ involving a Caputo  $\alpha$-time derivative and a power $\beta$ of the Laplacian when the dimension is large, $N> 4\beta$. The asymptotic profiles depend strongly on the space-time scale and on the time behavior of the spatial $L^1$ norm of the forcing term.
\end{abstract}

\maketitle

%%%%%%%%%%%%%%%%%%%%%%%%%%%%%%%%%%%%%%%%%%%%%%%%%%%%%%
\section{Introduction and main results}
\setcounter{equation}{0}

\subsection{Aim}
The purpose of this paper is to give a precise description of  the large-time behavior of solutions to the inhomogeneous \emph{fully nonlocal} heat equation
\begin{equation}\label{eq-f}
 	\partial_t^\a u+(-\Delta)^\beta u=f\quad\mbox{in }Q:=\R^N\times(0,\infty),\qquad
 	u(\cdot,0)=0\quad\mbox{in }\R^N,
\end{equation}
in the case of large dimensions, $N>4\beta$, completing the analysis started by Kemppainen, Siljander and Zacher in~\cite{Kemppainen-Siljander-Zacher-2017}.  Here, $\partial_t^\alpha$, $\alpha\in(0,1)$, denotes the so-called Caputo $\alpha$-derivative, introduced independently by many authors using different points of view, see for instance~\cite{Caputo-1967,Dzherbashyan-Nersesian-1968,Gerasimov-1948,Gross-1947,Liouville-1832,Rabotnov-1966}, defined for smooth functions by
$$
\displaystyle\partial_t^\alpha u(x,t)=\frac1{\Gamma(1-\alpha)}\,\partial_t\int_0^t\frac{u(x,\tau)- u(x,0)}{(t-\tau)^{\alpha}}\, {\rm d}\tau,
$$
and $(-\Delta)^\beta$, with $\beta\in(0,1]$,  is the usual $\beta$ power of the Laplacian, defined for smooth functions by
$(-\Delta)^{\beta}=\mathcal{F}^{-1}(|\cdot|^{2\beta}\mathcal{F})$, where $\mathcal{F}$ stands for Fourier transform; see for instance~\cite{Stein-book-1970}. Such equations, nonlocal both in space and time, are useful to model situations with long-range interactions and memory effects, and have been proposed for example to describe plasma  transport~\cite{delCastilloNegrete-Carreras-Lynch-2004,delCastilloNegrete-Carreras-Lynch-2005}; see also~\cite{Cartea-delCastilloNegrete-2007,Compte-Caceres-1998,Metzler-Klafter-2000,Zaslavsky-2002} for further models that use such equations.

Problem~\eqref{eq-f} does not have in general a classical solution, unless the forcing term $f$ is smooth enough. However, if $f\in L^\infty_{\rm loc}([0,\infty):L^1(\mathbb{R}^N))$,  it has a solution in a generalized sense,
defined by Duhamel's type formula
\begin{equation}\label{eq-formula}
 		u(x,t)=\int_0^t\int_{\R^N}Y(x-y,t-s)f(y,s)\,{\rm d}y{\rm d}s,
\end{equation}
with $Y=\partial_t^{1-\a} Z$, where $Z$ is the fundamental solution for the Cauchy problem,
\begin{equation}
\label{eq:Cauchy}
\partial_t^\a u+(-\Delta)^\beta u=0\quad\text{in }Q,\qquad u(\cdot,0)=u_0\quad\text{in }\mathbb{R}^N;
\end{equation}
see \cite{Eidelman-Kochubei-2004,Kemppainen-Siljander-Zacher-2017,Kochubei-1990}. Throughout the paper, by the solution to problem~\eqref{eq-f} we always mean the generalized solution given by~\eqref{eq-formula}.

The rate of decay/growth of the solution  depends  on the space-time scale under consideration, the $L^p$ norm with which we measure the size of $u$, and the size of the right-hand side $f$; see~\cite{Cortazar-Quiros-Wolanski-2022}, and also~\cite{Cortazar-Quiros-Wolanski-2021-Preprint} for the case of small dimensions. Our goal here is to determine, under some assumptions on the forcing term $f$, the asymptotic profile of the solution, once it is normalized taking into account the decay/growth rate. Let us point out that even for the local case, $\alpha=1$, $\beta=1$, such study is not yet complete; see Subsection~\ref{subsect-precedents} below.

\noindent\emph{Notation. } Let $g,h:\mathbb{R}_+\to\mathbb{R}_+$. In what follows we write $g\simeq h$ if there are constants $\nu,\mu>0$ such that $\nu \le g(t)/h(t)\le \mu$ for all $t\in\mathbb{R}_+$, and  $g\succ h$ if $\displaystyle\lim_{t\to\infty}\frac{g(t)}{h(t)}=\infty$.

\medskip

\subsection{The kernel $Y$. Critical exponents}\label{subsect-estimates W}

Our proofs depend on a good knowledge of the kernel $Y$, which, as mentioned above, is given by $Y=\partial_t^{1-\alpha}Z$. Let $\widehat{Z}=\widehat{Z}(\omega,t)$ denote the Fourier transform of the fundamental solution $Z$ of problem~\eqref{eq:Cauchy} in the $x$ variable. Then,
$$
\partial_t^\alpha \widehat{Z}(\omega,t)=-|\omega|^{2\beta}\widehat{Z}(\omega,t),\qquad \widehat{Z}(\omega,0)=1.
$$
The solution to this ordinary fractional differential equation is
$$
\widehat{Z}(\omega,t)=\mathbb{E}_\alpha(-|\omega|^{2\beta} t^\alpha),
$$
where $\mathbb{E}_\alpha$ is the Mittag-Leffler function of order $\alpha$,
$$
\mathbb{E}_\alpha(s)=\sum_{k=0}^\infty\frac{s^k}{\Gamma(1+k\alpha)}.
$$
Inverting the Fourier transform, we obtain that $Z$ has a self-similar form,
\begin{equation}
\label{eq:Z.selfsimilar}
Z(x,t)=t^{-N\theta}F(\xi),\quad \xi=x t^{-\theta},\quad \theta:=\frac{\alpha}{2\beta}
\end{equation}
with a radially symmetric positive profile $F$ that has an explicit expression in terms of certain  Fox's $H$-functions.
Hence, $Y$ has also a self-similar form,
\begin{equation}
\label{eq:Y.selfsimilar}
Y(x,t)=t^{-\sigma_*}G(\xi),\quad \xi=x t^{-\theta},\qquad \sigma_*:=1-\alpha+N\theta.
\end{equation}
Its profile $G$ is positive, radially symmetric, and smooth outside the origin, has integral 1, and,
in the case of large dimensions that we are considering here, $N>4\beta$, satisfies, for all $\beta\in(0,1]$, the estimates
\begin{align}
\label{eq:interior.estimate.profile}
 &G(\xi)\simeq 	{|\xi|^{4\beta-N}},&|\xi|\le 1,\\
%
%\label{eq-exterior estimates G1}
%&G(\xi)\simeq |\xi|^{\frac{(N-2)(\alpha-1)}{2-\alpha}}\exp({-\sigma|\xi|^{\frac2{2-\a}}}),&&|\xi|\ge1,\ \beta=1, \\
\label{eq:exterior.estimates.profile}
&G(\xi)=O\big(|\xi|^{-(N+2\beta)}\big),\quad |DG(\xi)|=O\big(|\xi|^{-(N+2\beta+1)}\big),&|\xi|\ge1.
\end{align}
We have also the limit
\begin{equation}
\label{eq:constant.origin}
|\xi|^{N-4\beta}G(\xi)\to\kappa\quad\text{as }|\xi|\to0\text{ for some constant }\kappa>0,
\end{equation}
which shows that the inner estimate~\eqref{eq:interior.estimate.profile} is sharp. The exterior estimates~\eqref{eq:exterior.estimates.profile} are also sharp if $\beta\in(0,1)$. In the special case $\beta=1$, both $G$ and $|DG|$ decay exponentially, but we do not need this fact in our calculations. All these estimates, and many others, are proved  in~\cite{Kemppainen-Siljander-Zacher-2017,Kim-Lim-2016}.

As a consequence of~\eqref{eq:interior.estimate.profile}--\eqref{eq:exterior.estimates.profile} we have the global bound
\begin{equation}
\label{eq:global.estimate.Y}
0\le Y(x,t)\le Ct^{-(1+\a)}|x|^{4\beta-N}\quad\text{in }Q,
\end{equation}
and  also the exterior bounds, valid if $|x|\ge \nu t^{\theta}$, $t>0$, for some $\nu>0$,
\begin{align}
  \label{eq:exterior.estimate.Y}
&0\le Y(x,t)\le C_\nu t^{2\a-1}|x|^{-(N+2\beta)},\\
\label{eq:bound.DY}
&|D Y(x,t)|\le C_\nu t^{2\a-1}|x|^{-(N+2\beta+1)},\\
\label{eq:bound.Yt}
&|\partial_t Y(x,t)|\le C_\nu t^{2\a-2}|x|^{-(N+2\beta)}.
\end{align}

Notice that $Y(\cdot,t)\in L^p(\mathbb{R}^N)$ if and only if $p\in [1,p_*)$, where $p_*:=N/(N-4\beta)$. Moreover,
\begin{equation}\label{eq:p.norm.Y}
\|Y(\cdot,t)\|_{L^p(\mathbb{R}^N)}=C t^{-\sigma(p)}\quad\text{for all }t>0\quad \text{if }p\in [1,p_*), \quad\text{where } \sigma(p):=\sigma_*-\frac{N\theta}p.
\end{equation}
Observe also that  $\sigma(p)<1$, and hence $Y\in L_{\rm loc}^1([0,\infty):L^p(\R^N))$, if and only if $p\in [1, p_{\rm c})$, with $p_{\rm c}:= N/(N-2\beta)$. Since the solution is given by a convolution of $f$ with $Y$ \emph{both in space and time}, the threshold value that will mark the border between subcritical and supercritical behaviors will be $p_{\rm c}$, and not $p_*$.

The self-similar form of $Y$, see~\eqref{eq:Y.selfsimilar}, stemming from the scaling invariance of the integro-differential operator,  gives a hint of the special role played by \emph{diffusive} scales, $|x|\simeq t^\theta$. As we will see, there is a marked difference between the behavior in compact sets and that in \emph{outer} scales, $|x|\ge \nu t^{\theta}$ for some $\nu>0$, with intermediate behaviors in \emph{intermediate} scales, $|x|\simeq \varphi(t)$, with $\varphi(t)\succ 1$, $\varphi(t)=o(t^\theta)$.

\subsection{Assumptions on $f$}
We always assume, no matter the space-time scale under consideration, the size hypothesis
 \begin{equation}\label{eq-hypothesis f}
\|f(\cdot,t)\|_{L^1(\mathbb{R}^N)}\le C(1+t)^{-\gamma} \quad\text{for some }\gamma\in\mathbb{R}\text{ and }C>0.
 \end{equation}
This condition guarantees that the function $u$ given by Duhamel's type formula~\eqref{eq-formula} is well defined, and moreover, that $u(\cdot,t)\in L^p (\mathbb{R}^N)$  for all $t>0$ in the \emph{subcritical} range $p\in[1,p_{\rm c})$, though not for $p\ge p_{\rm c}$.
In case we wish to analyze the large-time behavior of $u$ when $p$ is not subcritical, we will need some extra assumption on the spatial behavior of $f$ to force $u$ and the function giving its asymptotic behavior to be in the right space. The assumptions will depend on the scales, $p$, and $\gamma$.

\noindent\emph{Notation. } Given $p\ge p_{\rm c}$, we define
$$
q_{\rm c}(p):=\begin{cases}\displaystyle
\frac{Np}{2\beta p + N},&p\in[1,\infty),\\[8pt]
\displaystyle\frac{N}{2\beta},&p=\infty.
\end{cases}
$$

\noindent\textsc{Compact sets. } When dealing with the behavior in compact sets, if $p\ge p_{\rm c}$ we will assume that
\begin{equation}
\label{eq:assumption.f.p.not.subcritical}
\text{there is }q\in(q_{\rm c}(p),p]\text{ such that }\|f(\cdot,t)\|_{L^q(K)}\le C_K(1+t)^{-\gamma}\text{ for each }K\subset\subset \mathbb{R}^N.
\end{equation}
%This will be enough to prove that $u(\cdot,t)\in L^p_{\rm loc} (\mathbb{R}^N)$ for all $t>0$; see Lemma~\ref{lem:u.in.Lploc}.
On the other hand, if the time decay of the $L^1$ norm is not fast enough, namely, if $\gamma\le 1+\alpha$, in order to obtain a limit profile we will need to assume that $f$ is asymptotically a function in separate variables in the following precise sense, 				
\begin{equation}
\label{eq:hypothesis.asymp.sep.variables.1}
\text{there exists }g\in L^1(\R^N)\text{ such that }\|f(\cdot,t)(1+t)^\gamma-g\|_{L^1(\R^N)}\to0\quad\text{as }t\to\infty,
\end{equation}
again with an extra assumption if $p\ge p_{\rm c}$,
\begin{equation}
\label{eq:hypothesis.asymp.sep.variables.2}
g\in L_{\rm loc}^q(\mathbb{R}^N)\text{ for some }q\in(q_{\rm c}(p),p], \quad  \|f(\cdot,t)(1+t)^\gamma-g\|_{L^q(K)}\to 0\text{ as }t\to\infty\text{ if }K\subset\subset\R^N.
\end{equation}

\noindent\textsc{Intermediate scales. } For intermediate space-time scales, unless they are \emph{fast} (see Subsection~\ref{subsect-statement of results} for a precise definition) and $\gamma>1$, we have to assume that $f$ is asymptotically a function in separate variables, hypothesis~\eqref{eq:hypothesis.asymp.sep.variables.1}. If $\gamma=1$ and the scale is not \emph{slow} (see Subsection~\ref{subsect-statement of results} for a definition) we  will require
the tail control condition
\begin{equation}
\label{eq:uniform.tail.control.subcritical}
\sup_{t>0}\big( (1+t)^\gamma\|f(\cdot,t)\|_{L^1(\{|x|>R\})}\big)=o(1)\quad\text{as }R\to\infty.
\end{equation}

\noindent\emph{Remark. } Condition~\eqref{eq:uniform.tail.control.subcritical} is satisfied, for instance, if $|f(x,t)|\le h(x)(1+t)^{-\gamma}$,   for some $h\in L^1(\mathbb{R}^N)$.

Finally, if $p\ge p_{\rm c}$ we assume moreover the uniform tail control condition
\begin{equation}
\label{eq:tail.control.intermediate}
\sup_{t>0}\big( (1+t)^\gamma\|f(\cdot,t)\|_{L^q(\{|x|>R\})}\big)=O\big(R^{-N(1-\frac1q)}\big)\quad \text{as }R\to\infty\text{ for some }q\in(q_{\rm c}(p),p].
\end{equation}

\noindent\emph{Remark. } Condition~\eqref{eq:tail.control.intermediate} is satisfied, for instance, if $|f(x,t)|\le C |x|^{-N}(1+t)^{-\gamma}$ for some $C>0$.

\noindent\textsc{Outer scales. } For outer space-time scales and $\gamma\le 1$, we assume the uniform tail control condition~\eqref{eq:tail.control.intermediate} if $p$ is subcritical,
and
\begin{equation}
\label{eq:uniform.tail.control.not.subcritical}
\sup_{t>0}\big( (1+t)^\gamma\|f(\cdot,t)\|_{L^q(\{|x|>R)}\big)=o\big(R^{-N(1-\frac1q)}\big)\quad\text{for some }q\in(q_{\rm c}(p),p]\text{ as }R\to\infty
\end{equation}
otherwise.

\noindent\emph{Remark. } Condition~\eqref{eq:uniform.tail.control.not.subcritical} is satisfied, for instance, if $|f(x,t)|\le h(x)(1+t)^{-\gamma}$ with $g(x)=o(|x|^{-N})$.

We do not claim that the above conditions are optimal; but they are not too restrictive, and are easy enough to keep the proofs simple.

\subsection{Precedents}\label{subsect-precedents}

A full description of the large-time behavior of the homogeneous problem~\eqref{eq:fully.nl} for a nontrivial initial data $u_0\in L^1(\mathbb{R}^N)$ was recently given in~\cite{Cortazar-Quiros-Wolanski-2021a,Cortazar-Quiros-Wolanski-2021b}; see also~\cite{Kemppainen-Siljander-Zacher-2017}. The first precedent for the inhomogeneous problem~\eqref{eq-f} is~\cite{Kemppainen-Siljander-Zacher-2017}, where  the authors study the problem in the \emph{integrable in time} case $\gamma>1$ and  prove, for all $p\in [1,\infty]$ if $1\le N<2\beta$, and for  $p\in [1,p_{\rm c})$ otherwise,   that
\begin{equation}
\label{eq:result.KSZ.subcritical}
\lim_{t\to\infty}t^{\sigma(p)}\|u(\cdot,t)-M_\infty Y(\cdot,t)\|_{L^p(\mathbb{R}^N)}=0,\quad\text{where }M_\infty:=\int_0^\infty\int_{\mathbb{R}^N} f(x,t)\,{\rm d}x{\rm d}t<\infty.
\end{equation}
This result is also known to be valid for the local case, $\alpha=1$, $\beta=1$, if $p=1$; see~\cite{Biler-Guedda-Karch-2004,Dolbeault-Karch-2006}. In this special local situation $Y=Z$ is the well-known fundamental solution of the heat equation, whose profile does not have a spatial singularity and belongs to all $L^p$ spaces. But a complete analysis for $\alpha=1$, $\beta\in(0,1]$ is still missing,  and will be given elsewhere~\cite{Cortazar-Quiros-Wolanski-2022-Preprint}.

The above result \eqref{eq:result.KSZ.subcritical} \emph{cannot} hold when $N>4\beta$ if $p\ge p_{\rm c}$, even if we impose additional conditions on $f$ to guarantee that $u(\cdot,t)\in L^p(\mathbb{R}^N)$ and $\gamma>1$, since $Y(\cdot,t)\not\in L^p(\mathbb{R}^N)$ in that range. On the other hand, in the subcritical range $p\in[1,p_{\rm c})$, the result does not give information on the shape of the solution in \emph{inner regions}, that is, sets of the form $\{|x|\le g(t)\}$ with $g(t)=o(t^\theta)$, since $\|Y(\cdot,t)\|_{L^p(\{|x|\le g(t)\})}=o(t^{-\sigma(p)})$ in that case. We will tackle these two difficulties along the paper.

A first step towards the understanding of the large-time behavior of solutions to~\eqref{eq-f} in different space-time scales and for all possible values of $p$ was the determination of the decay/growth rates  under the above assumptions on $f$. This was done in~\cite{Cortazar-Quiros-Wolanski-2022} for the case of large dimensions that we are considering here, and in~\cite{Cortazar-Quiros-Wolanski-2021-Preprint} for low dimensions.

\subsection{Main results}\label{subsect-statement of results}

As we have already mentioned, the decay/growth rates of solutions and their asymptotic profiles depend on the space-time scale under consideration.

\noindent\textsc{Compact sets. } Given $\mu\in (0,N)$, and $h$ satisfying suitable integrability assumptions,  let
\begin{equation}
\label{eq:definition.Riesz.potential}
E_\mu(x)=|x|^{\mu-N},\quad
I_\mu[h](x)=\int_{\mathbb{R}^N}h(x-y)E_\mu(y)\,{\rm d}y.
\end{equation}
The large-time behavior in compact sets  will be described in terms of $I_{2\beta}[g]$ and $I_{4\beta}[F]$, where $g$ is the asymptotic spatial factor of the forcing term $f$, and
\begin{equation}
\label{eq:def.F}
F(x)=\int_0^\infty f(x,s)\,{\rm d}s.
\end{equation}

\noindent\emph{Remark. } Let $c_\mu:=\Gamma\big((N-\mu)/2\big)/(\pi^{N/2}2^\mu\Gamma(\mu/2))$. Then $c_\mu I_\mu[h]$ is the $\mu$-Riesz potential of $h$, so that $(-\Delta)^{\mu/2}(c_\mu I_\mu[h])=h$.

\begin{teo}\label{teo-compacts}  Let $f$ satisfy~\eqref{eq-hypothesis f}, and also~\eqref{eq:assumption.f.p.not.subcritical} if $p\ge p_{\rm c}$. If $\gamma\le 1+\a$ we  assume moreover~\eqref{eq:hypothesis.asymp.sep.variables.1}, and also~\eqref{eq:hypothesis.asymp.sep.variables.2} if $p\ge p_{\rm c}$. Let $u$ be solution to \eqref{eq-f}. Given $K\subset\subset \R^N$,
\begin{eqnarray}
\label{eq:teo-compacts}
&\|t^{\min\{\gamma,1+\alpha\}}u(\cdot,t)-\mathcal{L}\|_{L^p(K)}\to0 \quad\text{as }t\to\infty,\quad\text{where }
\\
\label{eq:profile.L}
&\mathcal{L}=
\begin{cases}
%\label{eq-compacts1},
c_{2\beta}I_{2\beta}[g]&\text{if }\gamma<1+\a,\\
%\label{eq-compacts3}
c_{2\beta}I_{2\beta}[g]+\kappa I_{4\beta}[F]&\text{if }\gamma=1+\a, \quad \text{with }\kappa \text{ as in~\eqref{eq:constant.origin},}\\
%\label{eq-compacts2}
\kappa I_{4\beta}[F]&\text{if }\gamma>1+\a.
\end{cases}
\end{eqnarray}
\end{teo}

\noindent\emph{Remarks.} (a) We already knew from~\cite{Cortazar-Quiros-Wolanski-2022} that $\|u(\cdot,t)\|_{L^p(K)}=O(t^{\min\{\gamma,1+\alpha\}})$ for any $K\subset\subset\mathbb{R}^N$. Theorem~\ref{teo-compacts} shows that this rate is sharp, $\|u(\cdot,t)\|_{L^p(K)}\simeq t^{\min\{\gamma,1+\alpha\}}$.

\noindent (b) Under the hypotheses of Theorem~\ref{teo-compacts}, if in addition $f(x,t)=g(x)(1+t)^{-\gamma}$, the asymptotic profile~$\mathcal{L}$ simplifies to
$$
\mathcal{L}=
\begin{cases}
c_{2\beta}I_{2\beta}[g]+\frac \kappa{\alpha}I_{4\beta}[g]&\text{if }\gamma=1+\a,\\
\frac \kappa{\gamma-1}I_{4\beta}[g]&\text{if }\gamma>1+\a.
\end{cases}
$$

\noindent (c) When the forcing term is independent of time, $f(\cdot,t)=g\in L^1(\R^N)$ for all $t>0$, Theorem \ref{teo-compacts} yields
\begin{equation}
\label{eq:convergence.stationary}
\|u(\cdot,t)-c_{2\beta}I_{2\beta}[g]\|_{L^p(K)}\to0\quad\mbox{as } t\to\infty
\end{equation}
for all $p\in[1,\infty]$ (assuming also $g\in L^q_{\rm loc}(\R^N)$ for some $q\in (q_{\rm c}(p),p]$ if $p\ge p_{\rm c}$). Hence, the limit profile in compact sets is a \emph{stationary} solution of the equation. This convergence result cannot be extended to the whole space if $p\in[1,p_{\rm c}]$, since $I_{2\beta}[g]\not\in L^p(\mathbb{R}^N)$ in this range.

The convergence result~\eqref{eq:convergence.stationary} for forcing terms independent of time also holds for solutions to
\begin{equation}
\label{eq:fully.nl}
\partial_t^\a u+(-\Delta)^\beta u=f\quad\mbox{in }Q,\qquad
 	u(\cdot,0)=u_0\quad\mbox{in }\R^N
\end{equation}
for any initial datum $u_0\in L^1(\R^N)$ (with the additional assumption $u_0\in L^q_{\rm loc}(\R^N)$ for some $q\in(q_{\rm c}(p),p]$ if $p\ge p_{\rm c}$). This follows from the linearity of the problem, since the generalized solution $v$ to~\eqref{eq:Cauchy} with $v(\cdot,0)=u_0$, given by $v(\cdot,t)=Z(\cdot,t)*u_0$, satisfies $\|v(\cdot,t)\|_{L^p(K)}\simeq t^{-\a}$ for every compact $K\subset\subset\mathbb{R}^N$; see~\cite{Cortazar-Quiros-Wolanski-2021a,Cortazar-Quiros-Wolanski-2021b},

\noindent (d) Under some integrability assumptions on $h$,
$$
I_\mu[h]\approx \Big(\int_{\mathbb{R}^N} h\Big)E_\mu\quad\text{as }|x|\to\infty;
$$
see Theorem~\ref{thm:behavior.Riesz.potential} in the Appendix for the details. Hence, the \lq\lq outer limit'', $|x|\to\infty$, of the function describing the large-time behavior in compact sets is given by
\begin{equation}
\label{eq:outer.limit.compact.sets}
\begin{aligned}
&t^{-\min\{\gamma,1+\alpha\}}\mathcal{L}(x)\approx\begin{cases}
t^{-\gamma}M_0 c_{2\beta}E_{2\beta}(x),&\gamma< 1+\alpha,\\
t^{-(1+\a)}M_\infty \kappa E_{4\beta}(x),&\gamma\ge 1+\alpha,
\end{cases}
\quad\text{as }|x|\to\infty, \\
&\text{where }
M_0=\int_{\mathbb{R}^N} g,\quad M_\infty=\int_0^\infty\int_{\mathbb{R}^N} f.
\end{aligned}
\end{equation}

\noindent\textsc{Intermediate scales. } These are scales for which $|x|\simeq \varphi(t)$, with $\varphi(t)\succ1$, $ \varphi=o(t^{\theta})$.
We will make a distinction among the different intermediate scales according to their speed, measured against the value of the decay/growth exponent $\gamma$. Thus, we have
\begin{align}
\label{eq:rates.slow}
\tag{S}
&\gamma<1; \text{ or } \gamma=1,\ \varphi(t)=o\big(t^\theta/(\log t)^{\frac1{2\beta}}\big);\text{ or } \gamma\in(1,1+\alpha),\  \varphi(t)=o\big(t^{\frac{1+\alpha-\gamma}{2\beta}}\big);\\
\label{eq:rates.critical.1}
\tag{C${}_1$}
&\gamma=1, \varphi(t)\simeq t^\theta/(\log t)^{\frac1{2\beta}}; \\
\label{eq:rates.critical}
\tag{C}
&\gamma\in(1,1+\alpha),\  \varphi(t)\simeq t^{\frac{1+\alpha-\gamma}{2\beta}};\\
\label{eq:rates.fast.1}
\tag{F${}_1$}
&\gamma=1,\ \varphi(t)\succ t^\theta/(\log t)^{\frac1{2\beta}};\\
\label{eq:rates.fast}
\tag{F}
&\gamma\in(1,1+\alpha),\ \varphi(t)\succ t^{\frac{1+\alpha-\gamma}{2\beta}}; \text{ or }
\gamma\ge 1+\alpha.
\end{align}
In \emph{slow} scales, satisfying~\eqref{eq:rates.slow}, the large-time behavior coincides with the outer limit of the behavior in compact sets, being given in terms of $E_{2\beta}$. Notice that when $\gamma<1$ all scales are slow. In \emph{fast} scales, satisfying~\eqref{eq:rates.fast} or~\eqref{eq:rates.fast.1}, the behavior coincides with the inner limit of the behavior in outer scales, and is given in terms of $E_{4\beta}$. Notice that when $\gamma\ge 1+\alpha$ all intermediate scales are fast. In the critical cases,~\eqref{eq:rates.critical.1} and~\eqref{eq:rates.critical}, the large time behavior involves both $E_{2\beta}$ and $E_{4\beta}$. The cases~\eqref{eq:rates.critical.1} and~\eqref{eq:rates.fast.1} (in both cases $\gamma=1$; that is the reason for the subscript) a factor $\log t$ is involved.

Let us recall from~\cite{Cortazar-Quiros-Wolanski-2022} that $\|u(\cdot,t)\|_{L^p(\{\nu<|x|/\varphi(t)<\mu\})}=O(\phi(t))$, where
\begin{equation}
\label{eq:rate.intermediate}
\phi(t)=\begin{cases}t^{-\gamma}\varphi(t)^{\frac{1-\sigma(p)}\theta}&\text{if \eqref{eq:rates.slow},~\eqref{eq:rates.critical.1}, or~\eqref{eq:rates.critical}},
\\
t^{-(1+\a)}\log t\,\varphi(t)^{\frac{1+\alpha-\sigma(p)}\theta}&\text{if \eqref{eq:rates.fast.1}},\\
t^{-(1+\alpha)}\varphi(t)^{\frac{1+\alpha-\sigma(p)}\theta}&\text{if \eqref{eq:rates.fast}}.
\end{cases}
\end{equation}
It is worth noticing that $\sigma(p)<1$ if and only if $p\in[1,p_{\rm c})$, and $\sigma(p)<1+\alpha$ if and only if $p\in[1,p_*)$.

\begin{teo}\label{teo-intermediate}   Let $\varphi(t)\succ1$, $ \varphi=o(t^{\theta})$. Let $f$ satisfy~\eqref{eq-hypothesis f}. We  assume moreover~\eqref{eq:hypothesis.asymp.sep.variables.1} if~\eqref{eq:rates.slow} or \eqref{eq:rates.critical} hold, and both~\eqref{eq:hypothesis.asymp.sep.variables.1} and~\eqref{eq:uniform.tail.control.subcritical}  if  \eqref{eq:rates.critical.1} or \eqref{eq:rates.fast.1} hold. When $p\ge p_{\rm c}$ we assume further~\eqref{eq:tail.control.intermediate}. Let $u$ be the solution to \eqref{eq-f}. Given $0<\nu<\mu<\infty$,
\begin{gather*}
\lim_{t\to\infty}\frac1{\phi(t)}\|u(\cdot,t)-\mathcal{L}(t)\|_{L^p(\{\nu<|x|/\varphi(t)<\mu\})}=0,\quad\text{with }\phi\text{ as in~\eqref{eq:rate.intermediate}, and}
\\
\mathcal{L}(t)
=\begin{cases}
t^{-\gamma}M_0 c_{2\beta}E_{2\beta}\quad&\text{if \eqref{eq:rates.slow}},\\
t^{-1}M_0c_{2\beta} E_{2\beta}+t^{-(1+\a)}\log t\, M_0\kappa  E_{4\beta}\quad&\text{if \eqref{eq:rates.critical.1}},\\
t^{-\gamma}M_0 c_{2\beta}E_{2\beta}+t^{-(1+\a)}M_\infty \kappa E_{4\beta}\quad&\text{if \eqref{eq:rates.critical}},\\
t^{-(1+\a)} \log t\,M_0\kappa  E_{4\beta}\quad&\text{if \eqref{eq:rates.fast.1}},\\
t^{-(1+\a)}M_\infty \kappa E_{4\beta}\quad&\text{if \eqref{eq:rates.fast}}.
\end{cases}
\end{gather*}
\end{teo}
\noindent\emph{Remarks. } (a) If $M_0,M_\infty\neq0$, then $\|\mathcal{L}(t)\|_{L^p(\{\nu<|x|/\varphi(t)<\mu\})}\simeq \phi(t)$ in all cases, since
\begin{equation}\label{eq:behaviour.powers}
\|E_{2\beta}\|_{L^p(\{\nu <|x|/\varphi(t)<\mu\})}\simeq\varphi(t)^{\frac{1-\sigma(p)}\theta},\qquad \|E_{4\beta}\|_{L^p(\{\nu <|x|/\varphi(t)<\mu\})}\simeq\varphi(t)^{\frac{1+\alpha-\sigma(p)}\theta}.
\end{equation}
As a corollary,  $\|u(\cdot,t)\|_{L^p(\{\nu<|x|/\varphi(t)<\mu\})}\simeq\phi(t)$.

\noindent (b) The behavior in \lq\lq inner'' intermediate scales is given by
$$
t^{-\gamma}M_0 c_{2\beta}E_{2\beta}\quad\text{if }\gamma<1+\alpha,\qquad t^{-(1+\a)}M_\infty \kappa E_{4\beta}\quad\text{if }\gamma\ge1+\alpha.
$$
This coincides with the \lq\lq outer limit'' of the behavior in compact sets; see~\eqref{eq:outer.limit.compact.sets}.

\medskip

\noindent\textsc{Exterior scales. }  These are scales for which $|x|\ge \nu t^{\theta}$, $\nu>0$. We already know from~\cite{Cortazar-Quiros-Wolanski-2022} that
\begin{equation}
\label{eq:rate.exterior}
\|u(\cdot,t)\|_{L^p(\{|x|>\nu t^\theta\})}=O(\phi(t)),\quad\text{where }
\phi(t)=\begin{cases}
t^{1-\gamma-\sigma(p)},&\gamma<1,\\
t^{-\sigma(p)}\log t,&\gamma=1,\\
t^{-\sigma(p)},&\gamma>1.
\end{cases}
\end{equation}
The asymptotic behavior of $u$ in such scales is given by a time convolution of $Y(\cdot,t)$ with the \lq\lq mass'' of $f$ at time $t$,
$$
M_f(t):=\int_{\mathbb{R}^N} f(x,t)\,{\rm d}x.
$$
\begin{teo}
\label{thm:outer.general}
Let $f$ satisfy~\eqref{eq-hypothesis f}. If $\gamma\le1$, assume moreover~\eqref{eq:uniform.tail.control.subcritical} if $p\in [1,p_{\rm c})$, and~\eqref{eq:uniform.tail.control.not.subcritical} if $p\ge p_{\rm c}$.
Then, given $\nu>0$,
$$
\lim_{t\to\infty}\frac1{\phi(t)}\Big\|u(\cdot,t)-\int_0^t M_f(s)Y(\cdot,t-s)\,{\rm d}s\Big\|_{L^p(\{|x|>\nu t^{\theta}\})}=0, \quad\text{with }\phi\text{ as in~\eqref{eq:rate.exterior}}.
$$
\end{teo}

\noindent\emph{Remarks. } (a) Notice that $|x|(t-s)^{-\theta}\ge \nu$ if $s\in(0,t)$ and $|x|>\nu t^\theta$. Therefore,
using~\eqref{eq:exterior.estimate.Y},
$$
\|Y(\cdot,t-s)\Big\|_{L^p(\{|x|>\nu t^{\theta}\})}\le Ct^{-\alpha-N\theta(1-\frac1p)}(t-s)^{2\alpha-1}\quad \text{for all }s\in(0,t).
$$
On the other hand,~\eqref{eq-hypothesis f} yields $|M_f(s)|\le C(1+s)^{-\gamma}$, and we conclude easily that
$$
\Big\|\int_0^t M_f(s)Y(\cdot,t-s)\,{\rm d}s\Big\|_{L^p(\{|x|>\nu t^{\theta}\})}\le C t^{-\alpha-N\theta(1-\frac1p)}\int_0^t (1+s)^{-\gamma}(t-s)^{2\alpha-1}\,{\rm d}s= O(\phi(t)).
$$
Hence, $\|u(\cdot,t)\|_{L^p(\{|x|>\nu t^{\theta}\})}=O(\phi(t))$.

\noindent (b) Assume that $M_f(s)(1+s)^\gamma\ge c$ or $M_f(s)(1+s)^\gamma\le -c$ for some $c>0$, a condition that is satisfied, for instance, if $f$ is of separate variables, $f(x,t)=g(x)(1+t)^{-\gamma}$, and $\int_{\mathbb{R}^N}g\neq0$. If $\beta\in(0,1)$, then $G(\xi)\simeq E_{-2\beta}(\xi)$ for $|\xi|>\nu>0$; see~\cite{Kemppainen-Siljander-Zacher-2017}. Therefore, for some constants $c_\nu>0$, which may change from line to line,
$$
\begin{aligned}
\Big\|\int_0^t M_f(s)Y(\cdot,t-s)\,{\rm d}s\Big\|_{L^p(\{|x|>\nu t^{\theta}\})}&\ge c_\nu\|E_{-2\beta}\|_{L^p(\{|x|>\nu t^{\theta}\})}
\int_0^t (1+s)^{-\gamma}(t-s)^{2\alpha-1}\,{\rm d}s\\
&\ge c_\nu
t^{-\alpha-N\theta(1-\frac1p)}\int_0^t (1+s)^{-\gamma}(t-s)^{2\alpha-1}\,{\rm d}s\simeq\phi(t).
\end{aligned}
$$
We conclude that
$$
\Big\|\int_0^t M_f(s)Y(\cdot,t-s)\,{\rm d}s\Big\|_{L^p(\{|x|>\nu t^{\theta}\})}\simeq\phi(t),
$$
and hence we have the sharp rate $\|u(\cdot,t)\|_{L^p(\{|x|>\nu t^{\theta}\})}\simeq \phi(t)$.

\medskip

When $\gamma\ge1$, we can avoid the time convolution in the description of the large time behavior, which is now given by $M(t)t^{1-\alpha}Y(\cdot,t)$, where $M(t)$ is the mass of the solution $u$ at time $t$,
$$
M(t):=\int_{\mathbb{R}^N}u(x,t)\,{\rm d}x.
$$
Let us remark that Fubini's theorem plus the fact that $\int_{\mathbb{R}^N}Y(x,t)\,{\rm d}x=t^{\alpha-1}$
yield
$$
M(t)
=\int_0^t M_f(s)(t-s)^{\a-1}\,{\rm d}s.
$$
Hence, $M(t)$ can be computed directly in terms of the forcing term $f$ without determining $u$.
\begin{teo}
\label{thm:outer.gamma.ge.1}
Under the assumptions of Theorem~\ref{thm:outer.general}, if $\gamma\ge1$, then
\[
\frac1{\phi(t)}\Big\|\int_0^t M_f(s)Y(\cdot,t-s)\,{\rm d}s-M(t)t^{1-\alpha}Y(\cdot,t)\Big\|_{L^p(\{|x|>\nu t^{\theta}\})}\to0\quad
\mbox{as }t\to\infty.
\]
As a corollary,
$$
\lim_{t\to\infty}\frac1{\phi(t)}\|u(\cdot,t)-M(t)t^{1-\alpha}Y(\cdot,t)\|_{L^p(\{|x|>\nu t^{\theta}\})}=0.
$$
\end{teo}
When $\gamma>1$ things simplify even further, since $M(t)t^{1-\alpha}$ has a computable limit.
\begin{teo}
\label{thm:outer.gamma.gt.1}
Let $\gamma>1$. Under the assumptions of Theorem~\ref{thm:outer.general}, $\displaystyle\lim\limits_{t\to\infty}M(t)t^{1-\alpha}=M_\infty$.
As a corollary,
\begin{equation}
\label{eq-exterior gamma larger 1}
\lim_{t\to\infty}t^{\sigma(p)}\|u(\cdot,t)-M_\infty Y(\cdot,t)\|_{L^p(\{|x|>\nu t^{\theta}\})}=0.
\end{equation}
\end{teo}

\noindent\emph{Remarks. } (a) The result~\eqref{eq-exterior gamma larger 1} was already known when $p\in[1,p_{\rm c})$; see~\cite{Kemppainen-Siljander-Zacher-2017}.

\noindent (b) Since $\|Y(\cdot,t)\|_{L^p(\{|x|>\nu t^{\theta}\})}=Ct^{-\sigma(p)}$ (see Subsection~\ref{subsect-estimates W}), if $\gamma>1$ and $M_\infty\neq0$  we obtain as a corollary that $\|u(\cdot,t)\|_{L^p(\{|x|>\nu t^{\theta}\})}\simeq t^{-\sigma(p)}$, without assuming that $f$ is of separate variables.

\noindent (c) When $\gamma>1$, the inner limit, $|\xi|\to 0$, of the outer profile is given by
$$
M_\infty Y(x,t)\approx M_\infty t^{-\sigma_*}\kappa E_{4\beta}(\xi)=t^{-(1+\alpha)}M_\infty \kappa E_{4\beta}(x),
$$
which coincides with the behavior for \lq\lq outer'' intermediate scales; see Theorem~\ref{teo-intermediate}, case~\eqref{eq:rates.fast}.

\medskip 	

The asymptotic limit can also be simplified when $\gamma=1$, at the expense of asking $f$ to be asymptotically of separate variables.
\begin{teo}
\label{thm:outer.gamma.equal.1}
Let $\gamma=1$. Under the assumption~\eqref{eq:hypothesis.asymp.sep.variables.1},
$M(t)=M_0t^{\alpha-1}\log t(1+o(1))$. As a corollary,
$$
\lim_{t\to\infty}\frac{t^{\sigma(p)}}{\log t}\|u(\cdot,t)-M_0\log t\,Y(\cdot,t)\|_{L^p(\{|x|>\nu t^{\theta}\})}=0.
$$
\end{teo}
\noindent\emph{Remarks. } (a) If $M_0\neq0$ and $\gamma=1$, we have the sharp decay rate $\|u(\cdot,t)\|_{L^p(\{|x|>\nu t^{\theta}\})}\simeq t^{-\sigma(p)}\log t$ assuming only that $f$ is asymptotically of separate variables.

\noindent (b) If $f$ satisfies~\eqref{eq:hypothesis.asymp.sep.variables.1} with  $\gamma=1$, the inner limit, $|\xi|\to 0$, of the outer profile is given by
$$
M_0\log t \,Y(x,t)\approx M_0 t^{-\sigma_*}\log t\, \kappa E_{4\beta}(\xi)=t^{-(1+\alpha)}\log t\,M_0 \kappa E_{4\beta}(x),
$$
which coincides with the behavior for \lq\lq outer'' intermediate scales; see Theorem~\ref{teo-intermediate}, case~\eqref{eq:rates.fast.1}.

\medskip

The purpose of our last result is to check that for $\gamma<1$ the inner limit of the outer profile also coincides with the behavior for \lq\lq outer'' intermediate scales, given by Theorem~\ref{teo-intermediate}, case~\eqref{eq:rates.slow},   when $f$ is asymptotically of separate variables.
\begin{prop}\label{prop-coherence}
Let $\gamma<1$. If $f$ satisfies~\eqref{eq:hypothesis.asymp.sep.variables.1} and the hypotheses of Theorem~\ref{thm:outer.general}, then
\[
\int_0^t M_f(s)Y(x,t-s)\,{\rm d}s=\big(1+o(1)\big)t^{-\gamma}M_0 c_{2\beta}E_{2\beta}(x)\quad\mbox{if }|\xi|=o(1)\text{ as }t\to\infty.
\]
\end{prop}

\noindent\emph{Remarks. } (a) The limit behavior~\eqref{eq:constant.origin} for the profile $G$ yields
\begin{equation}
\label{eq:behavior.Y.origin}
 Y(x,t)=(\kappa+o(1))t^{-(1+\a)}E_{4\beta}(x)\quad\mbox{if }|\xi|\to0.
\end{equation}
Hence, $M(t)t^{1-\alpha}Y(x,t)\approx M(t)t^{-2\alpha}\kappa E_{4\beta}(x)$ as $|\xi|\to0$. Therefore, if $\gamma<1$ the limit profile
in outer regions, $\int_0^t M_f(s)Y(\cdot,t-s)\,{\rm d}s$, does not coincide with $M(t)t^{1-\alpha}Y(x,t)$, in contrast with the case $\gamma\ge 1$.

\noindent (b) Under the assumptions of Proposition~\ref{prop-coherence}, if $M_0\neq0$ and $\delta>0$ is small enough, then
$$
\big|\int_0^t M_f(s)Y(x,t-s)\,{\rm d}s\big|\simeq  t^{-\gamma}E_{2\beta}(x) \quad\text{if }|x|t^{-\theta}<\delta.
$$
Therefore, if $\nu<\delta$, with $\delta>0$ small enough, we have
$$
\begin{aligned}
\Big\|\int_0^t M_f(s)Y(\cdot,t-s)\,{\rm d}s\Big\|_{L^p(\{|x|>\nu t^{\theta}\})}&\ge \Big\|\int_0^t M_f(s)Y(\cdot,t-s)\,{\rm d}s\Big\|_{L^p(\{\delta t^\theta>|x|>\nu t^{\theta}\})}\\
&\simeq t^{-\gamma}\|E_{2\beta}\|_{L^p(\{\delta t^\theta>|x|>\nu t^{\theta}\})}\simeq t^{1-\gamma-\sigma(p)}.
\end{aligned}
$$
We conclude that if $\gamma<1$, under suitable assumptions on $f$,
$$
\Big\|\int_0^t M_f(s)Y(\cdot,t-s)\,{\rm d}s\Big\|_{L^p(\{|x|>\nu t^{\theta}\})}\simeq t^{1-\gamma-\sigma(p)}\quad\text{if }\nu\text{ is small},
$$
and hence we have the sharp rate $\|u(\cdot,t)\|_{L^p(\{|x|>\nu t^{\theta}\})}\simeq t^{1-\gamma-\sigma(p)}$.

%%%%%%%%%%%%%%%%%%%%%%%%%%%%%%%%%%%%%%%%%%%%%%%%%%%%%%%%%%%%
\section{Compact sets}\label{sect-compacts}
\setcounter{equation}{0}

The goal of this section is to obtain the large-time behavior in compact sets, Theorem~\ref{teo-compacts}. We start by checking that, under the assumptions of that theorem, the functions giving the large-time behavior are in the desired spaces.
\begin{lema}
\label{lem:u.in.Lploc} {\rm (i)}  Let $g\in L^1(\R^N)$. If $p\ge p_{\rm c}$, assume in addition that $g\in L^q_{\rm loc}(\R^N)$ for some $q\in(q_{\rm c}(p),p]$. Then  $I_{2\beta}[g]\in L^p_{\rm loc}(\R^N)$.

\medskip

\noindent{\rm (ii)} Let $\gamma>1$. Let $f$ satisfy~\eqref{eq-hypothesis f}. If $p\ge p_{\rm c}$, assume in addition~\eqref{eq:assumption.f.p.not.subcritical}. Then $F$ given by~\eqref{eq:def.F} satisfies $I_{4\beta}[F]\in L^p_{\rm loc}(\R^N)$.
\end{lema}

\begin{proof} (i) We make the estimate $|I_{2\beta}[g]|\le {\rm I}+{\rm II}$, where
\[
{\rm I}(x)=\int_{|y|<1}|g(x-y)|E_{2\beta}(y)\,{\rm d}y,\qquad {\rm II}(x)=\int_{|y|>1}|g(x-y)|\,{\rm d}y.
\]
In order to estimate ${\rm I}$ we take
\begin{equation}
\label{eq:choice.q}
q=1\text{ if }p\in[1,p_{\rm c}),\quad q\in (q_{\rm c}(p),p] \text{ as in the hypothesis if }p\ge p_{\rm c},\qquad
1+\frac1p=\frac1q+\frac1r.
\end{equation}
Notice that $r\in [1,p_{\rm c})$.
%and hence $\sigma(r)\in(0,1)$.
%Since $p_{\rm c}<p_*$, we also have $r\in [1, p_*)$.
Then, using the hypotheses on $g$, we have
\[
\|{\rm I}\|_{L^p(K)}\le  \|g\|_{L^q(K+B_1)}\|E_{2\beta}\|_{L^r(B_1)}<\infty
\]
On the other hand, ${\rm II}(x)\le\|g\|_{L^1(\R^N)}$, hence the result.

\noindent (ii) Splitting the spatial integral as before, we have $|I_{4\beta}[F]|\le {\rm I}+{\rm II}$, where
$$
{\rm I}(x)=\int_0^\infty\int_{|y|<1}|f(x-y)|E_{4\beta}(y)\,{\rm d}y{\rm d}s,\qquad
{\rm II}(x)=\int_0^\infty\int_{|y|>1}|f(x-y)|\,{\rm d}y{\rm d}s.
$$
Taking  $q$ and $r$ as in~\eqref{eq:choice.q}, and using~\eqref{eq-hypothesis f}, and also~\eqref{eq:assumption.f.p.not.subcritical} if $p\ge p_{\rm c}$,
\[
\|{\rm I}\|_{L^p(K)}\le \|E_{4\beta}\|_{L^r(B_1)}\int_0^\infty \|f(\cdot,s)\|_{L^q(K+B_{1})}\,{\rm d}s
\le C\int_0^\infty (1+s)^{-\gamma}\,{\rm d}s<\infty,
\]
since $r\in [1,p_*)$, and $\gamma>1$. On the other hand, using the size hypothesis~\eqref{eq-hypothesis f},
${\rm II}(x)\le M_\infty$, whence the result.
\end{proof}
We now proceed to the proof of Theorem~\ref{teo-compacts}. As a first step we prove the result substituting the constant $c_{2\beta}$  in the definition~\eqref{eq:profile.L} of $\mathcal{L}$ by the constant
\begin{equation}
\label{eq:definition.A}
\displaystyle A:=\frac1{\theta\omega_N}\int_{\mathbb{R}^N} \frac{G(\xi)}{|\xi|^{2\beta}}\,{\rm d}\xi=\frac1\theta\int_0^\infty\rho^{N-1-2\beta}\tilde G(\rho)\,{\rm d}\rho,
\end{equation}
where $\omega_N$ denotes de measure of $\mathbb{S}^{N-1}:=\{x\in\mathbb{R}^N: |x|=1\}$, the unit sphere in $\mathbb{R}^N$, and $\tilde G(\rho)=G(\rho \zeta)$ for any $\zeta\in\mathbb{S}^{N-1}$ (remember that $G$ is radially symmetric).  We observe that the estimates~\eqref{eq:interior.estimate.profile}--\eqref{eq:exterior.estimates.profile} on the profile $G$ guarantee that $A$ is a finite number.

\begin{prop}
\label{prop:thm.with.A}
Under the assumptions of Theorem~\ref{teo-compacts}, the convergence result~\eqref{eq:teo-compacts} is true with the constant $c_{2\beta}$ in the definition~\eqref{eq:profile.L} of $\mathcal{L}$ substituted by the constant $A$ given in~\eqref{eq:definition.A}.
\end{prop}
\begin{proof}
As we will see, the value of $f$ at points that are far away from $x$ will not contribute to the behavior of the solution at that point. Hence, we estimate the error as
$$
\begin{aligned}
|t^{\min\{\gamma,1+\alpha\}}&u(x,t)-\mathcal{L}(x)|\le t^{\min\{\gamma,1+\alpha\}}( {\rm I}(x,t)+{\rm II}(x,t)),\quad\text{where}\\
{\rm I}(x,t)&=\int_0^t\int_{|y|>L}|f(x-y,t-s)| Y(y,s)\,{\rm d}y{\rm d}s,\\
{\rm II}(x,t)&=\Big|\int_0^t\int_{|y|<L}f(x-y,t-s)Y(y,s)\,{\rm d}y{\rm d}s
-t^{-\min\{\gamma,1+\alpha\}}\mathcal{L}(x)\Big|,
\end{aligned}
$$
with $L>0$ large to be chosen later.

In order to estimate ${\rm I}$, for every $t>1$ we split it as ${\rm I}={\rm I}_{1}+{\rm I}_{2}$, where
\[
\begin{aligned}
{\rm I}_{1}(x,t)&=\int_0^1\int_{|y|>L}|f(x-y,t-s)|Y(y,s)\,{\rm d}y{\rm d}s,\\
{\rm I}_{2}(x,t)&=\int_1^t\int_{|y|>L}|f(x-y,t-s)|Y(y,s)\,{\rm d}y{\rm d}s.
\end{aligned}
\]
We start with ${\rm I}_1$. Using the exterior bound~\eqref{eq:exterior.estimate.Y} and the decay assumption~\eqref{eq-hypothesis f},
\[\begin{aligned}
{\rm I}_{1}(x,t)&\le C\int_0^1\int_{|y|>L}|f(x-y,t-s)|\frac{s^{2\a-1}}{|y|^{N+2\beta}}\,{\rm d}y{\rm d}s\le \frac{Ct^{-\gamma}}{L^{N+2\beta}}\int_0^1s^{2\a-1}\,{\rm d}s<\ep t^{-\gamma}
\end{aligned}\]
for all $t>1$ if $L>0$ is large enough.

In order to estimate ${\rm I}_2$ we use now the global estimate~\eqref{eq:global.estimate.Y}, and again the decay assumption~\eqref{eq-hypothesis f} to obtain, for all $t>1$,
\[\begin{aligned}
{\rm I}_{2}(x,t)&\le C\int_1^t\int_{|y|>L}|f(x-y,t-s)|s^{-(1+\a)}E_{4\beta}(y)
\,{\rm d}y{\rm d}s\\
&\le \frac{Ct^{-\gamma}}{L^{N-4\beta}}\int_1^{t/2}s^{-(1+\alpha)}\,{\rm d}s+ \frac{Ct^{-(1+\alpha)}}{L^{N-4\beta}}\int_{t/2}^t(1+t-s)^{-\gamma}\,{\rm d}s\\
&\le \frac{Ct^{-\gamma}}{L^{N-4\beta}}+ \frac{Ct^{-(1+\alpha)}}{L^{N-4\beta}}\int_{t/2}^t(1+t-s)^{-\gamma}\,{\rm d}s.
\end{aligned}\]
Since
\begin{equation}
\label{eq:est.integral.t}
\int_{\lambda t}^t(1+t-s)^{-\gamma}\,{\rm d}s=\int_{0}^{(1-\lambda)t}(1+s)^{-\gamma}\,{\rm d}s\le C_\lambda
\begin{cases}
t^{1-\gamma},&\gamma<1,\\
\log t,&\gamma=1,\\
1,&\gamma>1,
\end{cases}
\end{equation}
for any $\lambda\in(0,1)$, then  ${\rm I}(x,t)\le C\ep t^{-\min\{\gamma,1+\alpha\}}$ for every $t>1$ if $L>0$ is large enough.

It will turn out that the values of $f$ at times close to $t$ only contribute to the asymptotic behavior of $u$ if $\gamma\le 1+\alpha$, while its values at times in the interval $(0,t/2)$ only matter if $\gamma\ge 1+\alpha$ (notice, however, that this interval expands to the whole $\mathbb{R}_+$ as $t\to\infty$).  Therefore, we make the estimate ${\rm II}\le{\rm II}_{1}+{\rm II}_{2}+{\rm II}_{3}$,  where
\[\begin{aligned}
	{\rm II}_{1}(x,t)&=\begin{cases}\displaystyle
\Big|\int_0^{\delta t}\int_{|y|<L}f(x-y,t-s)Y(y,s)\,{\rm d}y{\rm d}s- t^{-\gamma}AI_{2\beta}[g](x)\Big|,&\gamma\le 1+\alpha,\\
\displaystyle\int_0^{\delta t}\int_{|y|<L}|f(x-y,t-s)|Y(y,s)\,{\rm d}y{\rm d}s,&\gamma>1+\alpha,
\end{cases}
\\
{\rm II}_{2}(x,t)&=\int_{\delta t}^{t/2}\int_{|y|<L}|f(x-y,t-s)|Y(y,s)\,{\rm d}y{\rm d}s,\\
{\rm II}_{3}(x,t)&=\begin{cases}
\displaystyle
\int_{t/2}^t\int_{|y|<L}|f(x-y,t-s)|Y(y,s)\,{\rm d}y{\rm d}s,&\gamma<1+\alpha,\\
\displaystyle\Big|\int_{t/2}^t\int_{|y|<L}f(x-y,t-s)Y(y,s)\,{\rm d}y{\rm d}s-t^{-(1+\alpha)}\kappa I_{4\beta}[F](x)\Big|,&\gamma\ge1+\alpha,
\end{cases}
\end{aligned}
\]
for some  small value $\delta\in(0,1/2)$ to be chosen later.

We start by estimating ${\rm II}_1$ when $\gamma\le1+\a$. We have ${\rm II}_1\le \sum_{i=1}^4 {\rm II}_{1i}$, where
$$
\begin{aligned}
{\rm II}_{11}(x,t)&=
\int_0^{\delta t}\int_{|y|<L}|f(x-y,t-s)(1+t-s)^\gamma-g(x-y)|(1+t-s)^{-\gamma}Y(y,s) \,{\rm d}y{\rm d}s,\\
{\rm II}_{12}(x,t)&= \int_0^{\delta t}\big|(1+t-s)^{-\gamma}-t^{-\gamma}\big|\int_{|y|<L}|g(x-y)|Y(y,s)\,{\rm d}y{\rm d}s,\\
{\rm II}_{13}(x,t)&= t^{-\gamma}\int_{|y|<L}|g(x-y)|E_{2\beta}(y)\Big|\int_0^{\delta t}\frac{Y(y,s)}{E_{2\beta}(y)}\,{\rm d}s- A\Big|\,{\rm d}y,\\
{\rm II}_{14}(x,t)&=t^{-\gamma}A\int_{|y|>L}|g(x-y)|E_{2\beta}(y)\,{\rm d}y.
\end{aligned}
$$

We take  $q$ and $r$ as in~\eqref{eq:choice.q}. Using the $L^r$ norm of the kernel~\eqref{eq:p.norm.Y} if $s<1$, and the global estimate~\eqref{eq:global.estimate.Y}  when $s>1$,  together with the size assumption~\eqref{eq:hypothesis.asymp.sep.variables.1} on $f$ if $p$ is subcritical or~\eqref{eq:hypothesis.asymp.sep.variables.2} otherwise, for all $t>1/\delta$ we have
$$
\begin{aligned}
\|{\rm II}_{11}(\cdot,t)\|_{L^p(K)}&\le Ct^{-\gamma} \int_0^{1}\|f(\cdot,t-s)(1+t-s)^\gamma-g\|_{L^q(K+B_L)}\|Y(\cdot,s)\|_{L^r(B_L)}\,{\rm d}s\\
&\quad+Ct^{-\gamma}\int_1^{\delta t}\|f(\cdot,t-s)(1+t-s)^\gamma-g\|_{L^q(K+B_L)}s^{-(1+\alpha)}\|E_{4\beta}\|_{L^r(B_L)}\,{\rm d}s\\
&\le C\varepsilon t^{-\gamma}\Big(\int_0^{1}s^{-\sigma(r)}\,{\rm d}s+\int_1^{\delta t}s^{-(1+\alpha)}\,{\rm d}s\Big)\le C\varepsilon t^{-\gamma}.
\end{aligned}
$$

Given $\varepsilon>0$, there exists a small constant $\delta=\delta(\varepsilon)>0$ and a time $T_\varepsilon$ such that
\[
|(1+t-s)^{-\gamma}-t^{-\gamma}|<\ep t^{-\gamma}\quad\text{if }0<s<\delta t\text{ and }t\ge T_\varepsilon.
\]
Therefore, taking $q$ and $r$ as in~\eqref{eq:choice.q},
$$
\begin{aligned}
\|{\rm II}_{12}(\cdot,t)\|_{L^p(K)}&\le \varepsilon t^{-\gamma} \int_0^{1}\|g\|_{L^q(K+B_L)}\|Y(\cdot,s)\|_{L^r(B_L)}\,{\rm d}s\\
&\quad+\varepsilon t^{-\gamma} \int_1^{\delta t}\|g\|_{L^q(K+B_L)}s^{-(1+\alpha)}\|E_{4\beta}\|_{L^r(B_L)}\,{\rm d}s\le C\varepsilon t^{-\gamma}.
\end{aligned}
$$

As for ${\rm II}_{13}$, making the change of variables $\rho=|y|s^{-\theta}$ we get
\begin{equation}
\label{eq:convergence.to.A}
\int_0^{\delta t}\frac{Y(y,s)}{E_{2\beta}(y)}\,{\rm d}s=\frac1\theta\int_{|y|/(\delta  t)^\theta}^{\infty}\rho^{N-1-2\beta} G\Big(\frac{y}{|y|}\rho\Big)\,{\rm d}\rho.
\end{equation}
Therefore, due to the definition~\eqref{eq:definition.A} of $A$, for any fixed $L>0$,
\[
\Big|\int_0^{\delta t}\frac{Y(y,s)}{E_{2\beta}(y)}\,{\rm d}s- A\Big|=o(1) \quad\text{uniformly in }|y|<L.
\]
Hence, taking $q$ and $r$ as before,
$$
\|{\rm II}_{13}(\cdot,t)\|_{L^p(K)}\le o(t^{-\gamma}) \|g\|_{L^q(K+B_L)}\|E_{2\beta}\|_{L^r(B_L)}=o(t^{-\gamma}).
$$

We finally observe that
$$
{\rm II}_{14}(x,t)\le \frac{A\|g\|_{L^1(\mathbb{R}^N)}}{L^{N-2\beta}}t^{-\gamma}\le \varepsilon t^{-\gamma}
$$
if $L>0$ is large enough.

If $\gamma>1+\a$ the estimate for ${\rm II}_1$ is easier.
Let $q$ and $r$ be as above. Using the $L^p$ norm of the kernel~\eqref{eq:p.norm.Y} if $s<1$ and the global estimate~\eqref{eq:global.estimate.Y} when $s>1$, for all $t>1/\delta$ we have, thanks to the assumptions on the size of $f$,
$$
\begin{aligned}
\|{\rm II}_{1}(\cdot,t)\|_{L^p(K)}&\le \int_0^{1}\|f(\cdot,t-s)\|_{L^q(K+B_L)}\|Y(\cdot,s)\|_{L^r(B_L)}\,{\rm d}s\\
&\quad+\int_1^{\delta t}\|f(\cdot,t-s)\|_{L^q(K+B_L)}s^{-(1+\alpha)}\|E_{4\beta}\|_{L^r(B_L)}\,{\rm d}s\\
&\le C \int_0^{1}(1+t-s)^{-\gamma}s^{-\sigma(r)}\,{\rm d}s+\int_1^{\delta t}(1+t-s)^{-\gamma}s^{-(1+\alpha)}\,{\rm d}s\\
&\le C t^{-\gamma}\Big(\int_0^{1}s^{-\sigma(r)}\,{\rm d}s+\int_1^{\delta t}s^{-(1+\alpha)}\,{\rm d}s\Big)\le Ct^{-\gamma}.\\
\end{aligned}
$$

Now we turn our attention to  ${\rm II}_{2}$. Taking $q$ and $r$ as before,
$$
\begin{aligned}
\|{\rm II}_{2}(\cdot,t)\|_{L^p(K)}&\le\int_{\delta t}^{t/2}\|f(\cdot,t-s)\|_{L^q(K+B_L)}s^{-(1+\alpha)}\|E_{4\beta}\|_{L^r(B_L)}\,{\rm d}s\\
&\le C \int_{\delta t}^{t/2}(1+t-s)^{-\gamma}s^{-(1+\alpha)}\,{\rm d}s\le Ct^{-\gamma-\alpha}=o(t^{-\min\{\gamma,1+\alpha\}}).
\end{aligned}
$$

Finally, we turn our attention to ${\rm II}_{3}$. We start with the case $\gamma<1+\alpha$. Using the global estimate~\eqref{eq:global.estimate.Y} and then the size assumptions on $f$ we get, with $q$ and $r$ as above,
\[
\begin{aligned}
\|{\rm II}_{3}(\cdot,t)\|_{L^p(K)}
&\le\int_{t/2}^t\|f(\cdot,t-s)\|_{L^q(K+B_L)}s^{-(1+\alpha)}\|E_{4\beta}\|_{L^r(B_L)}\\
&\le Ct^{-(1+\alpha)}\int_{t/2}^t (1+t-s)^{-\gamma}\,{\rm d}s,
\end{aligned}
\]
from where we get, using~\eqref{eq:est.integral.t}, that
$\|{\rm II}_{3}(\cdot,t)\|_{L^p(K)}=t^{-\min\{\gamma,1+\alpha\}}o(1)$ in this range, as desired.

When $\gamma\ge 1+\alpha$ the estimate of ${\rm II}_3$ is more involved. We have ${\rm II}_3\le \sum_{i=1}^5 {\rm II}_{3i}$, where
$$
\begin{aligned}
{\rm II}_{31}(x,t)&=\int_{t/2}^t\int_{|y|<L}|f(x-y,t-s)|\big|Y(y,s)-\kappa E_{4\beta}(y)s^{-(1+\a)}\big|\,{\rm d}y{\rm d}s\\
{\rm II}_{32}(x,t)&=\kappa\int_{t/2}^{t-\sqrt{t}}\int_{|y|<L}|f(x-y,t-s)| E_{4\beta}(y)\big(s^{-(1+\a)}-t^{-(1+\a)}\big)\,{\rm d}y{\rm d}s\\
{\rm II}_{33}(x,t)&=\kappa\int_{t-\sqrt{t}}^t\int_{|y|<L}|f(x-y,t-s)| E_{4\beta}(y)\big(s^{-(1+\a)}-t^{-(1+\a)}\big)\,{\rm d}y{\rm d}s\\
{\rm II}_{34}(x,t)&=\kappa t^{-(1+\a)}\int_{|y|<L}E_{4\beta}(y)\Big|\int_{t/2}^tf(x-y,t-s)\,{\rm d}s-\int_0^\infty f(x-y,s)\,{\rm d}s\Big|\,{\rm d}y\\
{\rm II}_{35}(x,t)&=\kappa t^{-(1+\a)}\int_0^\infty \int_{|y|>L}|f(x-y,s)|E_{4\beta}(y)\,{\rm d}y{\rm d}s
\end{aligned}
$$

We first observe that
$$
|Y(y,s)-\kappa E_{4\beta}(y)s^{-(1+\a)}\big|=s^{-(1+\alpha)}E_{4\beta}(y) \Big|\frac{G(ys^{-\theta})}{E_{4\beta}(ys^{-\theta})}-\kappa\Big|.
$$
Therefore, thanks to~\eqref{eq:constant.origin}, given $\varepsilon>0$ and $L>0$ there is a time $T=T(\varepsilon,L)$ such that
$$
|Y(y,s)-\kappa E_{4\beta}(y)s^{-(1+\a)}\big|<\varepsilon s^{-(1+\alpha)}E_{4\beta}(y)\quad \text{if }|y|<L,\ t\ge T.
$$
Hence, if $t\ge T$ we have, for $q$ and $r$ chosen as above, and using~\eqref{eq:est.integral.t} and the size assumptions on $f$,
\[
\|{\rm II}_{31}(\cdot,t)\|_{L^p(K)}\le C\ep t^{-(1+\a)}\int_{t/2}^t\|f(\cdot,t-s)\|_{L^q(K+B_L)}\|E_{4\beta}\|_{L^r(B_L)}\,{\rm d}s\le C\varepsilon t^{-(1+\alpha)},
\]
so that $\|{\rm II}_{31}(\cdot,t)\|_{L^p(K)}\le C\varepsilon t^{-\min\{\gamma,(1+\alpha)\}}$, as desired.

On the other hand, with $q$ and $r$ as always, using the size assumptions on $f$,
\[\begin{aligned}
	\|{\rm II}_{32}(\cdot,t)\|_{L^p(K)}&\le C t^{-(1+\a)}\int_{t/2}^{t-\sqrt{t}}\|f(\cdot,t-s)\|_{L^q(K+B_L)}\|E_{4\beta}\|_{L^r(B_L)}\,{\rm d}s\\
	&\le C t^{-(1+\a)}\int_{t/2}^{t-\sqrt t}(1+t-s)^{-\gamma}\,{\rm d}s\le C t^{-(1+\a)}t^{\frac{1-\gamma}2}=o(t^{-(1+\a)}).
\end{aligned}\]

By the Mean Value Theorem, $0\le s^{-(1+\a)}-t^{-(1+\a)}\le (1+\alpha) s^{-(2+\alpha)}(t-s)$ if $s<t$. Therefore,
\[
\begin{aligned}
\|{\rm II}_{33}(\cdot,t)\|_{L^p(K)}&\le C t^{-(2+\a)}\sqrt t \int_{t-\sqrt t}^t \|f(\cdot,t-s)\|_{L^q(K+B_L)}\|E_{4\beta}\|_{L^r(B_L)}\,{\rm d}s\\
	&\le C t^{-(2+\a)}\sqrt t\int_0^{\sqrt t}(1+s)^{-\gamma}\,{\rm d}s\le C t^{-(1+\a)}t^{-1/2}=o(t^{-(1+\a)}).
\end{aligned}
\]

Now we notice that
$$
\Big|\int_{t/2}^tf(x-y,t-s)\,{\rm d}s-\int_0^\infty f(x-y,s)\,{\rm d}s\Big|=\Big|\int_{t/2}^\infty f(x-y,s)\,{\rm d}s\Big|.
$$
Therefore, for $t$ large enough,
\[
\begin{aligned}
\|{\rm II}_{34}(\cdot,t)\|_{L^p(K)}&\le Ct^{-(1+\alpha)}\int_{t/2}^\infty \|f(\cdot,s)\|_{L^q(K+B_L)}\|E_{4\beta}\|_{L^r(B_L)}\,{\rm d}s\\
&\le  Ct^{-(1+\alpha)}\int_{t/2}^\infty(1+s)^{-\gamma}\,{\rm d}s\le Ct^{-(\gamma+\alpha)}.
\end{aligned}\]

Finally,
\[
{\rm II}_{35}(x,t) \le \frac{Ct^{-(1+\alpha)}}{L^{N-4\beta}}\int_0^\infty\int_{\mathbb{R}^N} |f(y,s)|\,{\rm d}y{\rm d}s<\ep t^{-(1+\alpha)}
\]
if $L$ is large enough, and hence $\|{\rm II}_{35}(\cdot,t)\|_{L^p(K)}\le C\varepsilon t^{-(1+\alpha)}$.
\end{proof}

The proof of Theorem~\ref{teo-compacts} will then be complete if we are able to show that $A=c_{2\beta}$, a somewhat surprising result that is interesting on its own. The idea to prove this fact is to consider a particular forcing term for which the computations are \lq\lq explicit'', namely a stationary one, $g\in C^\infty_{\rm c}(\mathbb{R}^N)$, with $g$ nonnegative. We expect the solution $u$ of~\eqref{eq-f} with such a right-hand side to resemble for large times the stationary solution $S:=c_{2\beta}I_{2\beta}[g]$. Hence, we will study the difference, $U=S-u$. This function is a solution (in a generalized sense) to~\eqref{eq:Cauchy}, given by the formula $U(\cdot,t)=Z(\cdot,t)*S$. To check this last assert it is enough to observe that $S$ is a bounded classical solution to~\eqref{eq:fully.nl} with $f(\cdot,t)=g$ for all $t>0$ and $u_0=S$. But bounded classical solutions to~\eqref{eq:fully.nl} are unique, and represented by
$$
u(x,t)=\int_{\mathbb{R}^N}Z(x-y,t)u_0(y)\,{\rm d}y+\int_0^t\int_{\mathbb{R}^N}Y(x-y,t-s)f(y,s)\,{\rm d}y{\rm d}s,
$$
hence generalized solutions; see~\cite{Eidelman-Kochubei-2004,Kemppainen-Siljander-Zacher-2017,Kochubei-1990}.

Our first aim is to prove that $U$ vanishes asymptotically, so that $u$ indeed resembles $S$.
\begin{prop}
\label{prop:datum.newtonian.potential}
Let $U$ be the generalized solution to~\eqref{eq:Cauchy} with initial datum $U(\cdot,0)=c_{2\beta}I_{2\beta}[g]$ for some nonnegative $g\in C_{\rm c}^\infty(\mathbb{R}^N)$. Then, $\|U(\cdot,t)\|_{L^\infty(\mathbb{R}^N)}=o(1)$ as $t\to\infty$.
\end{prop}
\begin{proof}
We recall that the kernel $Z$ has a self-similar form; see~\eqref{eq:Z.selfsimilar}. Its profile $F$ belongs to $L^p(\mathbb{R}^N)$ if and only if $p\in[1,p_{\rm c})$; see, for instance,~\cite{Kemppainen-Siljander-Zacher-2017}. Hence, $\|Z(\cdot,t)\|_{L^p(\mathbb{R}^N)}=Ct^{-N\theta(1-\frac1p)}$ for $p$ in that range. On the other hand, $0\le U(x,0)\le C(1+|x|)^{-(N-2\beta)}$ (see, for instance, Theorem~\ref{thm:behavior.Riesz.potential} in the Appendix), so that $U(\cdot,0)\in L^p(\mathbb{R}^N)$ for all $p>p_{\rm c}$. Let $p>N/(2\beta)$.  This guarantees, on the one hand, that $p>p_c$, since $N>4\beta$, and on the other hand that $p/(p-1)<p_c$. Therefore, for any such~$p$, Young's inequality implies
$$
\|U(\cdot,t)\|_{L^\infty(\mathbb{R}^N)}\le \|Z(\cdot,t)\|_{L^{p/(p-1)}(\mathbb{R}^N)}\|U(\cdot,0)\|_{L^p(\mathbb{R}^N)}= C t^{-\frac{N\theta}p},
$$
whence the result.
\end{proof}
We are now ready to prove the equality of the constants, and hence the validity of Theorem~\ref{teo-compacts}.
\begin{coro}
The constant $A$ defined in~\eqref{eq:definition.A} coincides with $c_{2\beta}$. Therefore, Proposition~\ref{prop:thm.with.A} implies Theorem~\ref{teo-compacts}.
\end{coro}
\begin{proof}
Let $u$ be the solution to problem~\eqref{eq-f} with a non-negative and non-trivial right-hand side $g\in C^\infty_{\rm c}(\mathbb{R})$ independent of time. Then, as mentioned above, $U=S-u$ is a generalized solution to~\eqref{eq:Cauchy}. Thus, given $K\subset\subset\mathbb{R}^N$,
$$
\begin{aligned}
|c_{2\beta}-A|\|I_{2\beta}[g]\|_{L^\infty(K)}&=\|c_{2\beta}I_{2\beta}[g]-u(\cdot,t)+u(\cdot,t)-AI_{2\beta}[g]\|_{L^\infty(K)}\\
&\le \|U(\cdot,t)\|_{L^\infty(K)}+\|u(\cdot,t)-AI_{2\beta}[g]\|_{L^\infty(K)}\to 0\quad\text{as }t\to\infty,
\end{aligned}
$$
due to propositions~\ref{prop:thm.with.A} and~\ref{prop:datum.newtonian.potential}, hence the result, since $\|I_{2\beta}[g]\|_{L^\infty(K)}\neq0$.
\end{proof}

%%%%%%%%%%%%%%%%%%%%%%%%%%%%%%%%%%%%%%%%%%%%%%%%%%%%%%%%%%%%%%%%%%%%%%%%%%%%%%%%%%%%%%%%%%%%%%%
\section{Intermediate scales}\label{sect-intermediate}
\setcounter{equation}{0}

In this section we study the limiting profile in regions where $|x|\simeq \varphi(t)$ with $\varphi(t)\succ 1$ and $\varphi(t)=o(t^\theta)$ as $t\to\infty$.
\begin{proof}[Proof of Theorem~\ref{teo-intermediate}.]
We want to show that $u(\cdot,t)-\mathcal{L}(t)=o(\phi(t))$.
It will turn out that in these scales the large-time behavior at a point $x$ comes in first approximation from the behavior of $f(\cdot,t)$ at points that are relatively close to $x$, as compared with $|x|$. Hence, we estimate the error as $|u-\mathcal{L}|\le {\rm I}+{\rm II}+{\rm III}$, where
	\begin{equation*}
    \label{eq-decomposition intermediate1}
    \begin{aligned}
	{\rm I}(x,t)&=\Big|\int_{0}^t\int_{|x-y|\leq \ell(t) |x|} f(x-y,t-s)Y(y,s)\,{\rm d}y{\rm d}s-\mathcal{L}(t)(x)\Big|,\\
	{\rm II}(x,t)&=\int_0^{t/2}\int_{|x-y|>\ell(t) |x|} |f(x-y,t-s)|Y(y,s)\,{\rm d}y{\rm d}s,\\
	{\rm III}(x,t)&=\int_{t/2}^t\int_{|x-y|>\ell(t) |x|} |f(x-y,t-s)Y(y,s)|\,{\rm d}y{\rm d}s,
	\end{aligned}
	\end{equation*}
with $\ell(t)=o(1)$ such that $\ell(t)\in(0,1/2)$ for all $t>0$ to be further especified later.

The times that are closer to $t$ will contribute with a term involving $E_{2\beta}$ in slow and critical scales, while times which are closer to  $0$ will contribute with a term involving $E_{4\beta}$ in fast and critical scales. Therefore, we make an estimate of the form
${\rm I}\le {\rm I}_1+{\rm I}_2+{\rm I}_3$, where
\begin{equation*}\begin{aligned}
{\rm I}_1(x,t)&=\begin{cases}
\displaystyle\int_{0}^{t/2}\int_{|x-y|\leq \ell(t) |x|} |f(x-y,t-s)|Y(y,s)\,{\rm d}y{\rm d}s&\text{if \eqref{eq:rates.fast.1}/\eqref{eq:rates.fast}},\\
\displaystyle\Big|\int_{0}^{t/2}\int_{|x-y|\leq \ell(t) |x|} f(x-y,t-s)Y(y,s)\,{\rm d}y{\rm d}s-M_0At^{-\gamma} E_{2\beta}(x)\Big|&\text{if \eqref{eq:rates.slow}/\eqref{eq:rates.critical.1}/\eqref{eq:rates.critical}},
\end{cases}
\\
{\rm I}_2(x,t)&=\int_{t/2}^{t-t^{1-\delta(t)}}\int_{|x-y|\leq \ell(t) |x|} |f(x-y,t-s)|Y(y,s)\,{\rm d}y{\rm d}s,\\
{\rm I}_3(x,t)&=\begin{cases}\displaystyle\int_{t-t^{1-\delta(t)}}^t\int_{|x-y|\leq \ell(t) |x|} |f(x-y,t-s)|Y(y,s)\,{\rm d}y{\rm d}s,&\text{if \eqref{eq:rates.slow}},\\
\displaystyle\Big|\int_{t-t^{1-\delta(t)}}^t\int_{|x-y|\leq \ell(t) |x|} f(x-y,t-s)Y(y,s)\,{\rm d}y{\rm d}s\\
\hskip7.75cm-M_0\kappa t^{-(1+\alpha)}\log tE_{4\beta}(x)\Big|&\text{if \eqref{eq:rates.critical.1}/\eqref{eq:rates.fast.1}},\\
\displaystyle\Big|\int_{t-t^{1-\delta(t)}}^t\int_{|x-y|\leq \ell(t) |x|} f(x-y,t-s)Y(y,s)\,{\rm d}y{\rm d}s-M_\infty\kappa t^{-(1+\alpha)}E_{4\beta}(x)\Big|&\text{if \eqref{eq:rates.critical}/\eqref{eq:rates.fast}},
\end{cases}
\end{aligned}
\end{equation*}
with $\delta(t)\in (\log 2/\log t,1/2)$ to be further specified later. The lower bound for $\delta(t)$ guarantees that $t/2<t-t^{1-\delta(t)}$.

Since $\ell(t)\in (0,1/2)$, $\frac{|x|}2<|y|<\frac{3|x|}2$ if $|x-y|<\ell(t)|x|$. Hence, the global estimate~\eqref{eq:global.estimate.Y} yields
\begin{equation}
\label{eq:global.estimate.y.sim.x}
0\le Y(y,s)\le C s^{-(1+\alpha)}E_{4\beta}(x), \quad (y,s)\in Q,
\end{equation}
a bound that will used several times when estimating ${\rm I}_i$, $i=1,2,3$.

Let~\eqref{eq:rates.fast.1}/\eqref{eq:rates.fast} hold. If $\nu \varphi(t)<|x|<\mu\varphi(t)$, for large times we have on the one hand $|x|\neq0$, since $\varphi(t)\succ 1$, and on the other hand $(|x|/2)^{1/\theta}<t/2$ for large times, since $\varphi(t)=o(t^\theta)$. If $s<(|x|/2)^{1/\theta}$ and  $|y|>\frac{|x|}2$, we have $|y|>s^\theta$, and we may use the exterior estimate~\eqref{eq:exterior.estimate.Y} for the kernel. If $s>(|x|/2)^{1/\theta}$ we are away from the singularity in time, and we may use the global estimate~\eqref{eq:global.estimate.y.sim.x}. Therefore, using also the size assumption~\eqref{eq-hypothesis f} on $f$, we obtain
$$
\begin{aligned}
{\rm I}_1(x,t) &\leq C\int_{0}^{(|x|/2)^{1/\theta}}\int_{|y|>|x|/2} |f(x-y,t-s)|s^{2\alpha-1}E_{-2\beta}(y)\,{\rm d}y{\rm d}s
\\
&\quad +C\int_{(|x|/2)^{1/\theta}}^{t/2}\int_{|y|>|x|/2} |f(x-y,t-s)|s^{-(1+\alpha)}E_{4\beta}(y)\,{\rm d}y{\rm d}s\\
&\le C t^{-\gamma}\Big(E_{-2\beta}(x)\int_0^{(|x|/2)^{1/\theta}}s^{2\alpha-1}\,{\rm d}s+E_{4\beta}(x)
\int_{(|x|/2)^{1/\theta}}^{t/2}s^{-(1+\alpha)}\,{\rm d}s\Big)=O(t^{-\gamma})E_{2\beta}(x),
\end{aligned}
$$
which combined with~\eqref{eq:behaviour.powers} yields $\|{\rm I}_1(\cdot,t)\|_{L^p(\{\nu<|x|/ \varphi(t)<\mu\})}=O\big(t^{-\gamma}\varphi(t)^{\frac{1-\sigma(p)}\theta}\big)=o(\phi(t))$ if~\eqref{eq:rates.fast.1} or~\eqref{eq:rates.fast} hold.

When \eqref{eq:rates.slow}/\eqref{eq:rates.critical.1}/\eqref{eq:rates.critical} hold, we make the estimate ${\rm I}_1\le\sum_{i=1}^5{\rm I}_{1i}$, where
$$
\begin{aligned}
{\rm I}_{11}(x,t)&=
\int_0^{(|x|/2)^{1/\theta}}\int_{|x-y|<\ell(t)|x|}|f(x-y,t-s)(1+t-s)^\gamma-g(x-y)|(1+t-s)^{-\gamma}Y(y,s) \,{\rm d}y{\rm d}s,\\
{\rm I}_{12}(x,t)&=
\int_{(|x|/2)^{1/\theta}}^{t/2}\int_{|x-y|<\ell(t)|x|}|f(x-y,t-s)(1+t-s)^\gamma-g(x-y)|(1+t-s)^{-\gamma}Y(y,s) \,{\rm d}y{\rm d}s,\\
{\rm I}_{13}(x,t)&= \int_0^{\delta t}\big|(1+t-s)^{-\gamma}-t^{-\gamma}\big|\int_{|x-y|<\ell(t)|x|}|g(x-y)|Y(y,s)\,{\rm d}y{\rm d}s,\\
{\rm I}_{14}(x,t)&= \int_{\delta t}^{t/2}\big|(1+t-s)^{-\gamma}-t^{-\gamma}\big|\int_{|x-y|<\ell(t)|x|}|g(x-y)|Y(y,s)\,{\rm d}y{\rm d}s,\\
{\rm I}_{15}(x,t)&= t^{-\gamma}\int_{|x-y|<\ell(t)|x|}|g(x-y)|E_{2\beta}(y)\Big|\int_0^{t/2}\frac{Y(y,s)}{E_{2\beta}(y)}\,{\rm d}s- A\Big|\,{\rm d}y,\\
{\rm I}_{16}(x,t)&=t^{-\gamma}A\int_{|x-y|<\ell(t)|x|}|g(x-y)||E_{2\beta}(y)-E_{2\beta}(x)|\,{\rm d}y,\\
{\rm I}_{17}(x,t)&=t^{-\gamma}A E_{2\beta}(x)\int_{|x-y|\ge\ell(t)|x|}|g(x-y)|\,{\rm d}y.
\end{aligned}
$$

Since $s^\theta <|x|/2<|y|$ in the region of integration of ${\rm I}_{11}$,   we may use on the one hand the exterior estimate~\eqref{eq:exterior.estimate.Y} for the kernel, and on the other hand that $E_{-2\beta}(y)\le C E_{-2\beta}(x)$. Hence,
$$
{\rm I}_{11}(x,t) \le C t^{-\gamma}v(t)E_{-2\beta}(x)\int_0^{(|x|/2)^{1/\theta}}s^{2\alpha-1}\,{\rm d}s,
$$
with $v(t)=\sup_{\tau>t/2}\big\|f(\cdot,\tau)(1+\tau)^\gamma-g\|_{L^1(\R^N)}$.
But we know from~\eqref{eq:hypothesis.asymp.sep.variables.1} that $v(t)=o(1)$. Therefore,
${\rm I}_{11}(x,t)=o(t^{-\gamma})E_{2\beta}(x)$,
whence, using~\eqref{eq:behaviour.powers},
$$
\|{\rm I}_{11}(\cdot,t)\|_{L^p(\{\nu<|x|/ \varphi(t)<\mu\})}=o\big(t^{-\gamma}\varphi(t)^{\frac{1-\sigma(p)}\theta}\big)=o(\phi(t)).
$$

The region of integration in ${\rm I}_{12}$ avoids the singularity in time. Hence,  we may use the global estimate~\eqref{eq:global.estimate.y.sim.x} to obtain
$$
{\rm I}_{12}(x,t) \le C t^{-\gamma}v(t)E_{4\beta}(x)\int_{(|x|/2)^{1/\theta}}^{t/2}s^{-(1+\alpha)}\,{\rm d}s=o(t^{-\gamma})E_{2\beta}(x),
$$
whence $\|{\rm I}_{12}(\cdot,t)\|_{L^p(\{\nu<|x|/ \varphi(t)<\mu\})}=o(\phi(t))$.

To estimate ${\rm I}_{13}$ we note that  $\big|(1+t-s)^{-\gamma}-t^{-\gamma}\big|\le \varepsilon t^{-\gamma}$  if  $s\in(0,\delta t)$ with $\delta$ small and $t$ large. Therefore, changing the order of integration,
$$
{\rm I}_{13}(x,t)\le C\varepsilon t^{-\gamma} \int_{|x-y|<\ell(t)|x|}|g(x-y)|\int_0^{\delta t}Y(y,s)\,{\rm d}s{\rm d}y.
$$
But~\eqref{eq:definition.A} and~\eqref{eq:convergence.to.A} yield $\displaystyle\int_0^{\delta t}Y(y,s)\,{\rm d}s\le AE_{2\beta}(y)$. Hence, since $E_{2\beta}(y)\le CE_{2\beta}(x)$ for $|y|>|x|/2$, and using also the integrability of $g$,
$$
{\rm I}_{13}(x,t)\le C\varepsilon t^{-\gamma}E_{2\beta}(x) \int_{\frac12|x|<|y|<\frac32|x|}|g(x-y)|\,{\rm d}y\le C\varepsilon t^{-\gamma}E_{2\beta}(x),
$$
which combined with~\eqref{eq:behaviour.powers} yields $\|{\rm I}_{12}(\cdot,t)\|_{L^p(\{\nu<|x|/ \varphi(t)<\mu\})}=C\varepsilon t^{-\gamma}\varphi(t)^{\frac{1-\sigma(p)}\theta}\big)=C\varepsilon \phi(t))$.

In the estimate of ${\rm I}_{14}$  we may use once more the global estimate~\eqref{eq:global.estimate.y.sim.x}, since we are away from the singularity in time, to obtain
$$
{\rm I}_{14}(x,t) \le C t^{-\gamma}E_{4\beta}(x)\int_{\delta t}^{t/2}s^{-(1+\alpha)}\,{\rm d}s=O\big(t^{-(\gamma+\alpha)}\big)E_{4\beta}(x).
$$
Thus, using~\eqref{eq:behaviour.powers} and also that $\varphi(t)=o(t^\theta)$, we arrive at
$$
\|{\rm I}_{12}(\cdot,t)\|_{L^p(\{\nu<|x|/ \varphi(t)<\mu\})}=O\big(t^{-(\gamma+\alpha)}\varphi(t)^{\frac{1+\alpha-\sigma(p)}\theta}\big)=o\big(t^{-\gamma}\varphi(t)^{\frac{1-\sigma(p)}\theta}\big)=o(\phi(t)).
$$

As for ${\rm II}_{15}$, since $|y|\le 3|x|/2$ in the region of integration, it follows from~\eqref{eq:convergence.to.A} that
$$
\Big|\int_0^{t/2}\frac{Y(y,s)}{E_{2\beta}(y)}\,{\rm d}s- A\Big|<\int_0^{\frac{3|x|}{2(t/2)^\theta}}\rho^{N-1-2\beta} G\Big(\frac{y}{|y|}\rho\Big)\,{\rm d}\rho=o(1).
$$
Hence, using also that $|y|\ge |x|/2$ and the integrability of $g$, we get  ${\rm I}_{15}(x,t)= o(t^{-\gamma})E_{2\beta}(x)$, whence
$$
\|{\rm I}_{15}(\cdot,t)\|_{L^p(\{\nu<|x|/ \varphi(t)<\mu\})}=o\big(t^{-\gamma}\varphi(t)^{\frac{1-\sigma(p)}\theta}\big)=o(\phi(t)).
$$

If $|x-y|<\ell(t)|x|$, then  $\big||y|-|x|\big|\le |x-y|\le \ell(t)|x|$. Thus, using the Mean Value Theorem,
\begin{equation*}
\label{eq:difference.Ex.Ey}
|E_{2\beta}(y)-E_{2\beta}(x)|\le \frac{(N-2\beta)\ell(t)}{(1-\ell(t))^{N-2\beta+1}}E_{2\beta}(x)\le C\ell(t) E_{2\beta}(x).
\end{equation*}
Thus, ${\rm I}_{16}(x,t)=o(t^{-\gamma})E_{2\beta}(x)$, whence
$\|{\rm I}_{16}(\cdot,t)\|_{L^p(\{\nu<|x|/ \varphi(t)<\mu\})}=o\big(t^{-\gamma}\varphi(t)^{\frac{1-\sigma(p)}\theta}\big)=o(\phi(t))$.

As for ${\rm I}_{17}$, if $|x|>\nu\varphi(t)$ with $\varphi(t)\succ 1$, taking $\ell(t)\succ 1/\varphi(t)$,
$$
\int_{|x-y|\ge \ell(t)|x|}|g(x-y)|\,{\rm d}y\le \int_{|z|\ge \ell(t) \nu\varphi(t)}|g(z)|\,{\rm d}z=o(1)\quad \text{as }t\to\infty,
$$
since $g\in L^1(\mathbb{R}^N)$. Therefore, $\|{\rm I}_{17}(\cdot,t)\|_{L^p(\{\nu \le |x|/\varphi(t)\le \mu\})}=o(\phi(t))$.

Summarizing, if $\ell(t)=o(1)$ is such that $\ell(t)\succ 1/\varphi(t)$, then $\|{\rm I}_1(\cdot,t)\|_{L^p(\{\nu \le |x|/\varphi(t)\le \mu\})}=o(\phi(t))$.

We now turn our attention to ${\rm I}_{2}$. Using again~\eqref{eq:global.estimate.y.sim.x} and the assumption~\eqref{eq-hypothesis f} on the time decay of $f$, we have, since $s>t/2$ in the region of integration,
$$
{\rm I}_2(x,t)\le C t^{-(1+\alpha)}E_{4\beta}(x)\int_{t/2}^{t-t^{1-\delta(t)}}(1+t-s)^{-\gamma}\,{\rm d}s.
$$
On the other hand (remember that $\delta(t)<1/2$, so that $t^{1-\delta(t)}\to\infty$),
$$
\int_{t/2}^{t-t^{1-\delta(t)}}(1+t-s)^{-\gamma}\,{\rm d}s=\int_{t^{1-\delta(t)}}^{t/2}(1+s)^{-\gamma}\,{\rm d}s=\begin{cases}O(t^{1-\gamma}),&\gamma<1,\\
\log\frac{1+(t/2)}{1+t^{1-\delta(t)}}\le C\log(t^{\delta(t)}/2),&\gamma=1,\\
o(1),&\gamma>1
\end{cases}
$$
If $\delta(t)\succ 1/\log t$, then $t^{\delta(t)}\to\infty$, and hence $\log(t^{\delta(t)}/2)<C\log t^{\delta(t)}=C\delta(t)\log t=o(\log t)$ if $\delta(t)=o(1)$. With these additional assumptions on $\delta(t)$, we have then
$$
\|{\rm I}_2(\cdot,t)\|_{L^p(\{\nu <|x|/ \varphi(t)<\mu\})}=\begin{cases} O\big(t^{-(\gamma+\alpha)}\varphi(t)^{\frac{1+\alpha-\sigma(p)}\theta}\big)=o\big(t^{-\gamma}\varphi(t)^{\frac{1-\sigma(p)}\theta}\big),&\gamma<1,\\
o\big(t^{-(1+\alpha)}\log t\,\varphi(t)^{\frac{1+\alpha-\sigma(p)}\theta}\big),&\gamma=1,\\
o\big(t^{-(1+\alpha)}\varphi(t)^{\frac{1+\alpha-\sigma(p)}\theta}\big),&\gamma>1,
\end{cases}
$$
where we have used  that $\varphi(t)=o(t^\theta)$ in the last equality of the case $\gamma<1$. This estimate yields $\|{\rm I}_2(\cdot,t)\|_{L^p(\{\nu <|x|/ \varphi(t)<\mu\})}=o(\phi(t))$ in all cases.

As for ${\rm I}_3$, if \eqref{eq:rates.slow}, we use once more~\eqref{eq:global.estimate.y.sim.x} and~\eqref{eq-hypothesis f} to obtain, since $s>t/2$ in this case,
$$
\begin{aligned}
{\rm I}_3(x,t)&\le E_{4\beta}(x)\int_{t-t^{1-\delta(t)}}^t s^{-(1+\alpha)}(1+t-s)^{-\gamma}\,{\rm d}s \le Ct^{-(1+\a)}E_{4\beta}(x)\int_{t-t^{1-\delta(t)}}^t (1+t-s)^{-\gamma}\,{\rm d}s\\
&\le Ct^{-(1+\a)}E_{4\beta}(x)
\begin{cases}
O\big(t^{(1-\delta(t))(1-\gamma)}\big)=o(t^{1-\gamma}),&\gamma<1,\\
(1-\delta(t))\log t=O(\log t),&\gamma=1,\\
O(1),&\gamma>1,
\end{cases}
\end{aligned}
$$
where we have used that $\delta(t)\succ 1/\log t$ to show that $t^{-\delta(t)}=o(1)$ in the case $\gamma<1$. From here it is easily checked, using also that $\varphi(t)=o(t^\theta)$ when $\gamma<1$, that for all the cases included in~\eqref{eq:rates.slow} we have
$$
\|{\rm I}_3(\cdot,t)\|_{L^p(\{\nu<|x|/ \varphi(t)<\mu\})}=o\big(t^{-\gamma}\varphi(t)^{\frac{1-\sigma(p)}\theta}\big)=o(\phi(t)).
$$

If~\eqref{eq:rates.critical.1}/\eqref{eq:rates.fast.1}, we make the estimate ${\rm I}_3\le\sum_{i=1}^5{\rm I}_{3i}$, where
$$
\begin{aligned}
{\rm I}_{31}(x,t)&=
\int_{t-t^{1-\delta(t)}}^t\int_{|x-y|\leq \ell(t) |x|} |f(x-y,t-s)| |Y(y,s)-\kappa s^{-(1+\alpha)}E_{4\beta}(y)|\,{\rm d}y{\rm d}s,\\
{\rm I}_{32}(x,t)&=\kappa\int_{t-t^{1-\delta(t)}}^t\int_{|x-y|\leq\ell(t) |x|} |f(x-y,t-s)| s^{-(1+\alpha)}|E_{4\beta}(y)-E_{4\beta}(x)|\,{\rm d}y{\rm d}s,\\
{\rm I}_{33}(x,t)&=\kappa E_{4\beta}(x)\int_{t-t^{1-\delta(t)}}^t\int_{|x-y|\leq\ell(t) |x|} |f(x-y,t-s)| |s^{-(1+\alpha)}-t^{-(1+\alpha)}|\,{\rm d}y{\rm d}s,\\
{\rm I}_{34}(x,t)&=\kappa t^{-(1+\alpha)}E_{4\beta}(x)\Big|\int_{t-t^{1-\delta(t)}}^t\int_{|x-y|<\ell(t) |x|} f(x-y,t-s) \,{\rm d}y{\rm d}s- M_0\log(1+ t)\Big|,\\
{\rm I}_{35}(x,t)&=\kappa M_0 t^{-(1+\alpha)}E_{4\beta}(x)|\log(1+t)-\log t|.
\end{aligned}
$$
Since $|y|>|x|/2$ and $s> t/2$ in the integration region for ${\rm I}_{31}$,
$$
\begin{aligned}
{\rm I}_{31}(x,t)&=
\int_{t-t^{1-\delta(t)}}^t\int_{|x-y|\leq \ell(t) |x|} |f(x-y,t-s)| s^{-(1+\alpha)}E_{4\beta}(y)\Big|\frac{G(ys^{-\theta})}{E_{4\beta}(ys^{-\theta})}-\kappa\Big|\,{\rm d}y{\rm d}s\\
&\le C t^{-(1+\alpha)}E_{4\beta}(x)
\int_{t-t^{1-\delta(t)}}^t\int_{|x-y|\leq \ell(t) |x|} |f(x-y,t-s)| \Big|\frac{G(ys^{-\theta})}{E_{4\beta}(ys^{-\theta})}-\kappa\Big|\,{\rm d}y{\rm d}s.
\end{aligned}
$$
But $|y|\le 3|x|/2\le 3\mu\varphi(t)/2$ and $s>t/2$ imply that $|y|s^{-\theta}\le C\varphi(t)t^{-\theta}=o(1)$. Hence,~\eqref{eq:constant.origin} yields
$$
\Big|\frac{G(ys^{-\theta})}{E_{4\beta}(ys^{-\theta})}-\kappa\Big|=o(1)\quad\text{as }t\to\infty.
$$
Therefore, since $\gamma=1$ in this case, using the assumption~\eqref{eq-hypothesis f} on $f$,
$$
{\rm I}_{31}(x,t)= o(t^{-(1+\a)})E_{4\beta}(x)\int_{t-t^{1-\delta(t)}}^t(1+t-s)^{-1}\,{\rm d}s=  o\big(t^{-(1+\a)}\log t\big)E_{4\beta}(x).
$$

By the Mean Value Theorem, since $\big||y|-|x|\big|\le |x-y|\le \ell(t)|x|$ in the region of integration of ${\rm I}_{32}$,
\begin{equation*}
\label{eq:diferencias.Ex.Ey}
|E_{4\beta}(y)-E_{4\beta}(x)|\le \frac{(N-4\beta)\ell(t)}{(1-\ell(t))^{N-4\beta+1}}E_{4\beta}(x)\le C\ell(t) E_{4\beta}(x)
\end{equation*}
there, since $\ell(t)<1/2$. Thus, using also the assumption~\eqref{eq-hypothesis f},
\[
{\rm I}_{32}(x,t)\le C \ell(t) t^{-(1+\a)}E_{4\beta}(x) \int_{t-t^{1-\delta(t)}}^t(1+t-s)^{-1}\,{\rm d}s=o\big(t^{-(1+\a)}\log t\big)E_{4\beta}(x).
\]

In order to estimate ${\rm I}_{33}$, we apply the Mean Value Theorem to obtain that
$$
|s^{-(1+\alpha)}-t^{-(1+\alpha)}|\le C t^{-(1+\alpha)} t^{-\delta(t)}=o(t^{-(1+\alpha)})
$$
for all $s\in (t-t^{1-\delta(t)},t)$. Thus, using the assumption~\eqref{eq-hypothesis f},
$$
{\rm I}_{33}(x,t)= o(t^{-(1+\a)})E_{4\beta}(x)\int_{t-t^{1-\delta(t)}}^t(1+t-s)^{-1}\,{\rm d}s=  o\big(t^{-(1+\a)}\log t\big)E_{4\beta}(x).
$$

As for ${\rm I}_{34}$, we observe that
\[
\begin{aligned}
{\rm I}_{34}(x,t)&=\kappa t^{-(1+\alpha)}E_{4\beta}(x)\Big|\int_0^{t^{1-\delta(t)}}\int_{|y|<\ell(t)|x|} f(y,s)\,{\rm d}y{\rm d}s-\int_0^t\int_{\mathbb{R}^N} g(y)(1+s)^{-1}\,{\rm d}y{\rm d}s\Big|\\
&\le \kappa t^{-(1+\alpha)}E_{4\beta}(x)\Big(\int_0^{t^{\delta(t)}}\|f(\cdot,s)\|_{L^1(\R^N)}\,{\rm d}s+\|g\|_{L^1(\R^N)}\int_0^{t^{\delta(t)}}(1+s)^{-1}\,{\rm d}s\\
&\quad+ \int_{t^{\delta(t)}}^{t^{1-\delta(t)}}\sup_{s>t^{\delta(t)}}\|f(\cdot,s)(1+s)-g\|_{L^1(\R^N)}(1+s)^{-1}\,{\rm d}s\\
&\quad+\int_{|y|>\ell(t){|x|}}|g(y)|\int_{t^{\delta(t)}}^{t^{1-\delta(t)}}(1+s)^{-1}\,{\rm d}s+\|g\|_{L^1(\R^N)}\int_{t^{1-\delta(t)}}^t(1+s)^{-1}\,{\rm d}s\Big).
\end{aligned}
\]
Notice that $t^{\delta(t)}<t^{1-\delta(t)}$, since $\delta(t)<1/2$. From~\eqref{eq:hypothesis.asymp.sep.variables.1} we get that
$$
\sup_{s>t^{\delta(t)}}\|f(\cdot,s)(1+s)-g\|_{L^1(\R^N)}=o(1),
$$
since, due to the condition $\delta(t)\succ 1/\log t$, $t^{\delta(t)}\to\infty$ as $t\to\infty$. Moreover, if $|x|\ge \nu \varphi(t)$,
$$
\int_{|y|>\ell(t){|x|}}|g(y)|\,{\rm d}y \le \int_{|y|>\ell(t)\nu\varphi(t)}|g(y)|\,{\rm d}y=o(1),
$$
since $g$ is integrable and $\ell(t)\varphi(t)\to\infty$ (remember that $\ell(t)\succ 1/\varphi(t)$). Therefore, using also the assumption~\eqref{eq-hypothesis f},
$$
{\rm I}_{34}(x,t)
\le C t^{-(1+\alpha)}E_{4\beta}(x)(\delta(t)\log t+o(1)\log t))=o(t^{-(1+\alpha)}\log t)E_{4\beta}(x),
$$
since $\delta(t)=o(t)$.

It is immediate to check that
$$
{\rm I_{35}}(x,t)=\kappa M_0 t^{-(1+\alpha)}\log t E_{4\beta}(x)\big|\frac{\log(1+t)}{\log t}-1|=o(t^{-(1+\a)}\log t)E_{4\beta}(x).
$$

Summarizing,  ${\rm I}_3(x,t)=o(t^{-(1+\a)}\log t)E_{4\beta}(x)$, whence
$$
\|{\rm I}_3(\cdot,t)\|_{L^p(\{\nu <|x|/ \varphi(t)<\mu\})}=o\big(t^{-(1+\a)}\log t\,\varphi(t)^{\frac{1+\alpha-\sigma(p)}\theta}\big),
$$
from where it is easily checked that $\|{\rm I}_3(\cdot,t)\|_{L^p(\{\nu <|x|/ \varphi(t)<\mu\})}=o(\phi(t))$ if~\eqref{eq:rates.critical.1}/\eqref{eq:rates.fast.1}.

If~\eqref{eq:rates.critical}/\eqref{eq:rates.fast},  we make the estimate ${\rm I}_3\le\sum_{i=1}^6{\rm I}_{3i}$, with ${\rm I}_{3i}$, $i=1,2,3$ as for~\eqref{eq:rates.critical.1}/\eqref{eq:rates.fast.1}, and
$$
\begin{aligned}{\rm I}_{34}(x,t)&=\kappa t^{-(1+\alpha)}E_{4\beta}(x)\int_0^{t-t^{1-\delta(t)}}\int_{|x-y|\leq \ell(t) |x|} |f(x-y,t-s)| \,{\rm d}y{\rm d}s,\\
{\rm I}_{35}(x,t)&=\kappa t^{-(1+\alpha)}E_{4\beta}(x)\int_0^t\int_{|x-y|> \ell(t) |x|} |f(x-y,t-s)| \,{\rm d}y{\rm d}s,\\
{\rm I}_{36}(x,t)&=\kappa t^{-(1+\alpha)}E_{4\beta}(x)\int_t^\infty\int_{\mathbb{R}^N}|f(y,s)|\,{\rm d}s.
\end{aligned}
$$
Reasoning as for the cases~\eqref{eq:rates.critical.1}/\eqref{eq:rates.fast.1}, we get (notice that now $\gamma>1$),
$$
{\rm I}_{3i}(x,t)=o(t^{-(1+\a)})E_{4\beta}(x)\int_{t-t^{1-\delta(t)}}^t(1+t-s)^{-1}\,{\rm d}s=  o(t^{-(1+\a)})E_{4\beta}(x),\quad i=1,2,3.
$$
On the other hand, using the hypothesis~\eqref{eq-hypothesis f} on $f$,
\[{\rm I}_{34}(x,t)\le Ct^{-(1+\alpha)}E_{4\beta}(x)\int_0^{t-t^{1-\delta(t)}}(1+t-s)^{-\gamma}\,{\rm d}s
\le C t^{-(\gamma+\alpha)}E_{4\beta}(x).
\]
Finally, as $f\in L^1(Q)$ and $\ell(t)\succ 1/\phi(t)$,
\[
\begin{aligned}
&\int_0^{t}\int_{{|x-y|> \ell(t) |x|}} |f(x-y,t-s)|\,{\rm d}y{\rm d}s\le
\int_0^{\infty}\int_{{|y|> \ell(t)\nu\varphi(t) }} |f(y,s)|\,{\rm d}y{\rm d}s=o(1) \quad\text{for }|x|>\nu\varphi(t),\\
&\int_{t}^\infty\int_{\mathbb{R}^N} |f(y,s)|\,{\rm d}y{\rm d}s=o(1).
\end{aligned}
\]
Therefore, ${\rm I}_{3i}(x,t)= o(t^{-(1+\alpha)})E_{4\beta}(x)$, $i=1,6$.

Summarizing,  ${\rm I}_3(x,t)=o(t^{-(1+\a)})E_{4\beta}(x)$, whence
$$
\|{\rm I}_3(\cdot,t)\|_{L^p(\{\nu <|x|/ \varphi(t)<\mu\})}=o\big(t^{-(1+\a)}\varphi(t)^{\frac{1+\alpha-\sigma(p)}\theta}\big),
$$
from where it is easily checked that $\|{\rm I}_3(\cdot,t)\|_{L^p(\{\nu <|x|/ \varphi(t)<\mu\})}=o(\phi(t))$ if~\eqref{eq:rates.critical}/\eqref{eq:rates.fast}.

Now we analyze ${\rm II}$. We make the decomposition ${\rm II}={\rm II}_1+{\rm II}_2+{\rm II}_3$, where
\begin{equation*}\label{eq-decomposition intermediate2}
\begin{aligned}
{\rm II}_1(x,t)&= \int_0^{ t/2}\int_{\stackrel{|y|<k(t)|x|}{|x-y|>\ell(t) |x|}}|f(x-y,t-s)|Y(y,s)\,{\rm d}y{\rm d}s\\
{\rm II}_2(x,t)&= \int_0^{(k(t)|x|)^{1/\theta}}\int_{\stackrel{|y|>k(t)|x|}{|x-y|>\ell(t) |x|}}|f(x-y,t-s)|Y(y,s)\,{\rm d}y{\rm d}s\\
{\rm II}_3(x,t)&= \int_{(k(t)|x|)^{/\theta}}^{ t/2}\int_{\stackrel{|y|>k(t)|x|}{|x-y|>\ell(t) |x|}}|f(x-y,t-s)|Y(y,s)\,{\rm d}y{\rm d}s,
\end{aligned}
\end{equation*}
with $k(t)=o(1)$ such that $k(t)\in(0,1/2)$ for all $t>0$ to be further especified later.

We estimate ${\rm II}_{1}$ as  ${\rm II}_1\le{\rm II}_{11}+{\rm II}_{12}$, where
\begin{equation*}\label{eq-decomposition II2}\begin {aligned}
{\rm II}_{11}(x,t)&=\int_0^{t/2}\int_{|y|<\min\{k(t)|x|,s^{\theta}\}}|f(x-y,t-s)|Y(y,s)\,{\rm d}y{\rm d}s,\\
{\rm II}_{12}(x,t)&=\int_0^{(k(t)|x|)^{1/\theta}}\int_{s^\theta<|y|<k(t)|x|}|f(x-y,t-s)|Y(y,s)\,{\rm d}y{\rm d}s.
\end{aligned}\end{equation*}
Since $k(t)<1/2$,  $|x-y|\ge|x|/2>\nu \varphi(t)/2$. Hence, taking $q$ and $r$ as in~\eqref{eq:choice.q}, and using the global bound~\eqref{eq:global.estimate.Y},
\[
\begin{aligned}
\|{\rm II}_{11}(\cdot,t)\|_{L^p(\{\nu<|x|/ \varphi(t)<\mu\})}&\le\int_0^{t/2} \|f(\cdot,t-s)\|_{L^q(\{|x|>\frac\nu2\varphi(t)\}}  \|Y(\cdot,s)\|_{L^r(\{|x|<\min\{\mu k(t)\varphi(t),s^{\theta}\}\})}\,{\rm d}s\\
&\le C m(t)t^{-\gamma}\int_0^{t/2}   s^{-(1+\alpha)}(\min\{\mu k(t)\varphi(t),s^{\theta}\})^{\frac{1+\alpha-\sigma(r)}{\theta}}\,{\rm d}s,
\end{aligned}
\]
where $m(t):=\sup_{\tau>0}\|f(\cdot,\tau)(1+\tau)^\gamma\|_{L^q(\{|x|>\frac\nu2\varphi(t)\})}$, since $\|E_{4\beta}\|_{L^r(\{|x|<a\})}\le Ca^{\frac{1+\alpha-\sigma(r)}\theta}$. Thanks to assumptions~\eqref{eq-hypothesis f} and~\eqref{eq:tail.control.intermediate},
$$
m(t)=O\big(\varphi(t)^{-N\big(1-\frac1q\big)}\big)
=O\big(\varphi(t)^{\frac{\sigma(r)-\sigma(p)}\theta}\big),
$$
and hence, since $k(t)=o(1)$ and $\sigma(r)<1$,
\[
\begin{aligned}
\|{\rm II}_{11}&(\cdot,t)\|_{L^p(\{\nu<|x|/ \varphi(t)<\mu\})}\\
&\le Cm(t)t^{-\gamma}\Big( \int_0^{(k(t)\mu\varphi(t))^{1/\theta}}  s^{-\sigma(r)} \,{\rm d}s +(k(t)\mu\varphi(t))^{\frac{1+\alpha-\sigma(r)}\theta}\int_{(k(t)\mu\varphi(t)))^{1/\theta}}^{t/2}  s^{-(1+\a)} \,{\rm d}s\Big)\\
&\le Cm(t)t^{-\gamma} (k(t)\varphi(t))^{\frac{1-\sigma(r)}\theta} \le k(t)^{\frac{1-\sigma(r)}\theta}O\big(t^{-\gamma}\varphi(t)^{\frac{1-\sigma(p)}\theta}\big)=o\big(t^{-\gamma}\varphi(t)^{\frac{1-\sigma(p)}\theta}\big)
=o(\phi(t)).
\end{aligned}
\]

As for ${\rm II}_{12}$, since $|y|/s^\theta>1$ in the integration range, we may use the exterior estimate~\eqref{eq:exterior.estimate.Y} for the kernel to obtain
$$
{\rm II}_{12}(x,t)\le C\int_0^{(k(t)|x|)^{1/\theta}}\int_{s^\theta<|y|<k(t)|x|}|f(x-y,t-s)|s^{2\a-1}E_{-2\beta}(y)\,{\rm d}y{\rm d}s.
$$
On the other hand, since $k(t)<1/2$,  $|x-y|\ge|x|/2>\nu \varphi(t)/2$. Hence, taking $q$ and $r$ as in~\eqref{eq:choice.q},
\[
\begin{aligned}
\|{\rm II}_{12}(\cdot,t)\|_{L^p(\{\nu<|x|/ \varphi(t)<\mu\})}
&\le
C\int_0^{(\mu k(t)  \varphi(t))^{1/\theta}}\|f(\cdot,t-s)\|_{L^q(\{|x|>\frac\nu2\varphi(t)\})} s^{2\a-1}\|E_{-2\beta}\|_{L^r(\{|x|>s^{\theta}\})}\,{\rm d}s\\
&\le C m(t)t^{-\gamma}\int_0^{(\mu k(t)  \varphi(t)))^{1/\theta}}s^{-\sigma(r)}\,{\rm d}s\le k(t)^{\frac{1-\sigma(r)}\theta}O\big(t^{-\gamma}\varphi(t)^{\frac{1-\sigma(p)}\theta}\big)
\\
&=o(\phi(t)).
\end{aligned}
\]

In order to estimate ${\rm II}_2$, we observe that in the region of integration $s^\theta<k(t) |x|<|y|$.  Therefore, we may use the outer estimate~\eqref{eq:exterior.estimate.Y}, and hence
$$
{\rm II}_{2}(x,t)\le C\int_0^{(k(t)|x|)^{1/\theta}}\int_{\stackrel{|y|>s^\theta}{|x-y|>\ell(t) |x|}}|f(x-y,t-s)|s^{2\a-1}E_{-2\beta}(y)\,{\rm d}y{\rm d}s.
$$
Therefore, taking $q$ and $r$ as in~\eqref{eq:choice.q},
\[
\begin{aligned}
\|{\rm II}_{2}(\cdot,t)\|_{L^p(\{\nu<|x|/ \varphi(t)<\mu\})}
&\le
C\int_0^{(\mu k(t)  \varphi(t))^{1/\theta}}\|f(\cdot,t-s)\|_{L^q(\{|x|>\nu\ell(t)\varphi(t)\})} s^{2\a-1}\|E_{-2\beta}\|_{L^r(\{|x|>s^{\theta}\})}\,{\rm d}s\\
&\le C n(t)t^{-\gamma}\int_0^{(\mu k(t)  \varphi(t)))^{1/\theta}}s^{-\sigma(r)}\,{\rm d}s= C n(t)t^{-\gamma}(k(t)\varphi(t))^{\frac{1-\sigma(r)}\theta}
\end{aligned}
\]
where $n(t):=\sup_{\tau>0}\|f(\cdot,\tau)(1+\tau)^\gamma\|_{L^q(\{|x|>\nu\ell(t)\varphi(t)\})}$. Thanks to the assumptions~\eqref{eq-hypothesis f} and~\eqref{eq:tail.control.intermediate},
$$
n(t)=O\big((\ell(t)\varphi(t))^{-N(1-\frac1q)}\big)
=\ell(t)^{-N(1-\frac1q)}O\big(\varphi(t)^{\frac{\sigma(r)-\sigma(p)}\theta}\big),
$$
whence, if $\ell(t)$, in addition to $\ell(t)\succ 1/\varphi(t)$, satisfies also $\ell(t)\succ k(t)^{(1-\sigma(r))/(N\theta(1-\frac1q))}$,
$$
\begin{aligned}
\|{\rm II}_{2}(\cdot,t)\|_{L^p(\{\nu<|x|/ \varphi(t)<\mu\})}&= \ell(t)^{-N(1-\frac1q)}k(t)^{\frac{1-\sigma(r)}\theta}O\big(t^{-\gamma}\varphi(t)^{\frac{1-\sigma(p)}\theta}\big)\\
&
=o\big(t^{-\gamma}\varphi(t)^{\frac{1-\sigma(p)}\theta}\big)=o(\phi(t)).
\end{aligned}
$$

To estimate ${\rm II}_3$, we use the global bound~\eqref{eq:global.estimate.Y}. Then, if $|x|>\nu \varphi(t)$,
\[
\begin{aligned}
{\rm II}_3(x,t)&\le C \int_{(k(t)|x|)^{1/\theta}}^{ t/2}s^{-(1+\alpha)}\int_{\stackrel{|y|>k(t)|x|}{|x-y|>\ell(t) |x|}}|f(x-y,t-s)|E_{4\beta}(y)\,{\rm d}y{\rm d}s
\\
&\le C k(t)^{4\beta-N}t^{-\gamma}E_{4\beta}(x) \int_{(k(t)|x|)^{1/\theta}}^{ t/2}s^{-(1+\alpha)}(1+t-s)^\gamma\| f(\cdot,t-s)\|_{L^1(\{|x|>\nu\ell(t)\varphi(t)\})}\,{\rm d}y{\rm d}s\\
&\le  C v(t) k(t)^{4\beta-N}E_{4\beta}(x)t^{-\gamma}\int_{(k(t)|x|)^{1/\theta}}^{t/2}s^{-(1+\a)}\,{\rm d}s=Cv(t) k(t)^{2\beta-N}t^{-\gamma}E_{2\beta}(x)
\end{aligned}
\]
where $v(t):=\sup_{\tau>t/2}\|(1+\tau)^\gamma f(\cdot,\tau)\|_{L^1(\{|x|>\nu\ell(t)\varphi(t)\})}$ is a bounded function, thanks to the size assumption~\eqref{eq-hypothesis f}. Using~\eqref{eq:behaviour.powers},
$$
\|{\rm II}_3(\cdot,t)\|_{L^p(\{\nu<|x|/ \varphi(t)<\mu\})}=Cv(t) k(t)^{2\beta-N}t^{-\gamma}\varphi(t)^{\frac{1-\sigma(p)}\theta}.
$$

Since $v$ is bounded, in fast scales~\eqref{eq:rates.fast}  we have, see~\eqref{eq:rate.intermediate},
$$
\|{\rm II}_3(\cdot,t)\|_{L^p(\{\nu<|x|/ \varphi(t)<\mu\})}=C \phi(t) k(t)^{2\beta-N}t^{1+\alpha-\gamma}\varphi(t)^{-2\beta}.
$$
On the other hand it is readily checked that in these scales $t^{1+\alpha-\gamma}\varphi(t)^{-2\beta}=o(1)$. Therefore, if we take $k(t)\succ \big(t^{1+\alpha-\gamma}\varphi(t)^{-2\beta}\big)^{1/(N-2\beta)}$, we finally arrive at $\|{\rm II}_3(\cdot,t)\|_{L^p(\{\nu<|x|/ \varphi(t)<\mu\})}=o(\phi(t))$, as desired.

For the scales~\eqref{eq:rates.slow}, \eqref{eq:rates.critical.1}, \eqref{eq:rates.critical}, and \eqref{eq:rates.fast.1} we use assumption~\eqref{eq:hypothesis.asymp.sep.variables.1} to show that, since $\ell(t)\succ 1/\varphi(t)$,
$$
v(t)\le \sup_{\tau>t/2}\|(1+\tau)^\gamma f(\cdot,\tau)-g\|_{L^1(\{|x|>\nu\ell(t)\varphi(t)\})}+\|g\|_{L^1(\{|x|>\nu\ell(t)\varphi(t)\})}=o(1).
$$
Therefore, if we take $k(t)\succ v(t)^{1/(N-2\beta)}$, we get
$$
\|{\rm II}_3(\cdot,t)\|_{L^p(\{\nu<|x|/ \varphi(t)<\mu\})}=o\big(t^{-\gamma}\varphi(t)^{\frac{1-\sigma(p)}\theta}\big)=o(\phi(t)).
$$
	
We now consider the last term, ${\rm III}$. We have ${\rm III}={\rm III}_1+{\rm III}_2$, where
\[\begin{aligned}
{\rm III}_1(x,t)&=\int_{t/2}^t\int_{\stackrel{|y|<h(t)|x|}{|x-y|>\ell(t) |x|}}|f(x-y,t-s)|Y(y,s)\,{\rm d}y{\rm d}s,\\
{\rm III}_2(x,t)&= \int_{t/2}^t\int_{\stackrel{|y|>h(t)|x|}{|x-y|>\ell(t) |x|}}|f(x-y,t-s)|Y(y,s)\,{\rm d}y{\rm d}s,
\end{aligned}
	\]
with $h(t)=o(1)$ such that $h(t)\in(0,1/2)$ for all $t>0$ to be further especified later.

Since $h(t)<1/2$,  $|x-y|\ge|x|/2>\nu \varphi(t)/2$. Hence, taking $q$ and $r$ as in~\eqref{eq:choice.q}, and using the global estimate~\eqref{eq:global.estimate.Y},
\[
\begin{aligned}
	\|{\rm III}_1(\cdot,t)\|_{L^p(\{\nu<|x|/ \varphi(t)<\mu\})}&\le
\int_{t/2}^t\|f(\cdot,t-s)\|_{L^q(\{|x|>\frac\nu2\varphi(t)\})}\|Y(\cdot,s)\|_{L^r(\{|x|<h(t)\mu  \varphi(t)\})}\,{\rm d}s
\\
&\le  Cm(t)t^{-(1+\a)}\|E_{4\beta}\|_{L^r(\{|x|<h(t)\mu  \varphi(t)\})}\int_{t/2}^t(1+t-s)^{-\gamma}\,{\rm d}s,
\end{aligned}
\]
with $m(t)$ as above. Then, since $\sigma(r)<1$ and $h(t)=o(1)$, and  using also~\eqref{eq:est.integral.t}, we conclude that
$$
\begin{aligned}
\|{\rm III}_1(\cdot,t)\|_{L^p(\{\nu<|x|/ \varphi(t)<\mu\})}&\le Ct^{-(1+\a)}h(t)^{\frac{1+\alpha-\sigma(r)}\theta} \varphi(t)^{\frac{1+\alpha-\sigma(p)}\theta}\int_{t/2}^t(1+t-s)^{-\gamma}\,{\rm d}s\\
&=o\big(\varphi(t)^{\frac{1+\alpha-\sigma(p)}\theta}\big)\int_{t/2}^t(1+t-s)^{-\gamma}\,{\rm d}s=o(\phi(t)).
\end{aligned}
$$

To estimate ${\rm III}_2$, we use the global bound~\eqref{eq:global.estimate.Y}. Then, if $|x|> \nu \varphi(t)$,
\[
{\rm III}_2(x,t)\le C \int_{t/2}^t\int_{\stackrel{|y|>h(t)|x|}{|x-y|>\ell(t) |x|}}|f(x-y,t-s)|s^{-(1+\a)}E_{4\beta}(y)\,{\rm d}y{\rm d}s\le C h(t)^{4\beta-N}t^{-(1+\a)}w(t)E_{4\beta}(x),
\]
where $w(t)=\displaystyle\int_0^{t/2}\|f(\cdot,\tau)\|_{L^1(\{|x|\ge \ell(t)\nu\varphi(t)\})}\,{\rm d}\tau$, so that, using~\eqref{eq:behaviour.powers},
$$
\|{\rm III}_2(\cdot,t)\|_{L^p(\{\nu<|x|/ \varphi(t)<\mu\})}\le C h(t)^{4\beta-N}w(t)t^{-(1+\a)}\varphi(t)^{\frac{1+\alpha-\sigma(p)}\theta}
$$
If $\gamma>1$, then $w(t)=o(1)$. Hence, taking $\displaystyle h(t)\succ w(t)^{1/(N-4\beta)}$,
$$
\|{\rm III}_2(\cdot,t)\|_{L^p(\{\nu<|x|/ \varphi(t)<\mu\})}=o\big(t^{-(1+\a)}\varphi(t)^{\frac{1+\alpha-\sigma(p)}\theta}\big),
$$
whence it is easy to check that $\|{\rm III}_2(\cdot,t)\|_{L^p(\{\nu<|x|/ \varphi(t)<\mu\})}=o(\phi(t))$
in the scales \eqref{eq:rates.slow}, \eqref{eq:rates.critical}, and \eqref{eq:rates.fast} when $\gamma>1$.

If $\gamma<1$, something which only happens in the case~\eqref{eq:rates.slow}, hypothesis~\eqref{eq-hypothesis f} yields $w(t)=O(t^{1-\gamma})$.  Remember that $\varphi(t)=o(t^\theta)$. Hence, taking $h(t)\succ (\varphi(t)/t^\theta)^{\alpha/(\theta(N-4\beta))}=o(1)$,
$$
\|{\rm III}_2(\cdot,t)\|_{L^p(\{\nu<|x|/ \varphi(t)<\mu\})}\le C h(t)^{4\beta-N}(\varphi(t)/t^\theta)^{\frac\alpha\theta} t^{-\gamma}\varphi(t)^{\frac{1-\sigma(p)}\theta}
=o\big(t^{-\gamma}\varphi(t)^{\frac{1-\sigma(p)}\theta}\big)=o(\phi(t)).
$$

If $\gamma=1$, the size hypothesis~\eqref{eq-hypothesis f} yields $w(t)=O(\log t)$. If $\varphi(t)=o\big(t^\theta/(\log t)^{\frac1{2\beta}}\big)$, then, taking $h(t)\succ (\varphi(t)/(t^\theta/(\log t)^{\frac1{2\beta}}))^{\frac\alpha{\theta(N-4\beta)}}=o(1)$,
$$
\begin{aligned}
\|{\rm III}_2(\cdot,t)\|_{L^p(\{\nu<|x|/ \varphi(t)<\mu\})}&=C h(t)^{4\beta-N}(\varphi(t)/(t^\theta/(\log t)^{\frac1{2\beta}}))^{\frac\alpha\theta}t^{-1}\varphi(t)^{\frac{1-\sigma(p)}\theta}\\
&=o\big(t^{-1}\varphi(t)^{\frac{1-\sigma(p)}\theta}\big)=o(\phi(t)),
\end{aligned}
$$
which completes the analysis of the case~\eqref{eq:rates.slow}.

For the remaining cases with $\gamma=1$, namely~\eqref{eq:rates.critical.1} and~\eqref{eq:rates.fast.1}, we require the tail control hypothesis~\eqref{eq:uniform.tail.control.subcritical}, that yields $w(t)=o(\log t)$. Taking $h(t)\succ (w(t)/\log t)^{1/(N-4\beta)}$,
for~\eqref{eq:rates.fast.1} we have
$$
\begin{aligned}
\|{\rm III}_2(\cdot,t)\|_{L^p(\{\nu<|x|/ \varphi(t)<\mu\})}&\le C h(t)^{4\beta-N}\frac{w(t)}{\log t}t^{-(1+\alpha)}\log t\,\varphi(t)^{\frac{1+\alpha-\sigma(p)}\theta}\\
&
=o\big(t^{-(1+\alpha)}\log t\,\varphi(t)^{\frac{1+\alpha-\sigma(p)}\theta}\big)=o(\phi(t)),
\end{aligned}
$$
and in the case~\eqref{eq:rates.critical.1}, for which $\varphi(t)\simeq t^\theta/(\log t)^{\frac1{2\beta}}$,
$$
\begin{aligned}
\|{\rm III}_2(\cdot,t)\|_{L^p(\{\nu<|x|/ \varphi(t)<\mu\})}&\le  C h(t)^{4\beta-N}\frac{w(t)}{\log t}t^{-1}\varphi(t)^{\frac{1-\sigma(p)}\theta}(\varphi(t)(\log t)^{\frac1{2\beta}}/t^\theta)^{\frac\alpha\theta}
\\
&\simeq h(t)^{4\beta-N}\frac{w(t)}{\log t}t^{-1}\varphi(t)^{\frac{1-\sigma(p)}\theta}
=o\big(t^{-1}\varphi(t)^{\frac{1-\sigma(p)}\theta}\big)=o(\phi(t)).
\end{aligned}
$$

\end{proof}

%%%%%%%%%%%%%%%%%%%%%%%%%%%%%%%%%%%%%%%%%%%%%%%%%%%%%%%%%%%%%%%%%%%%%%%%%%%%%%%%
\section{Exterior regions}
\setcounter{equation}{0}

This section is devoted to prove the results concerning the large-time behavior in exterior regions, $\{|x|>\nu t^{\theta}\}$ for some $\nu>0$, theorems~\ref{thm:outer.general}--\ref{thm:outer.gamma.equal.1} and Proposition~\ref{prop-coherence}.

\begin{proof}[Proof of Theorem~\ref{thm:outer.general}.]
We make the decomposition
 $$
 \begin{aligned}
 \big|u(x,t)-\int_0^t &M_f(s)Y(x,t-s)\,{\rm d}s\big|\le {\rm I}(x,t)+{\rm II}(x,t),\quad\text{where}\\
 {\rm I}(x,t)&=\int_0^t\int_{|y|<\delta|x|} |f(y,s)| |Y(x-y,t-s)-Y(x,t-s)|\,{\rm d}y{\rm d}s,\\
 {\rm II}(x,t)&=\int_0^t\int_{|y|>\delta|x|} |f(y,s)||Y(x-y,t-s)-Y(x,t-s)|\,{\rm d}y{\rm d}s,
 \end{aligned}
 $$
with $\delta\in(0,1/2)$ to be fixed later.

By the Mean Value Theorem, for each $x,y,t$ and $s$ there is a value $\lambda\in(0,1)$ such that
$$
 |Y(x-y,t-s)-Y(x,t-s)|=|DY(x-\lambda y,t-s)||y|.
$$
But, if $|y|<\delta|x|$ with $\delta\in(0,1/2)$ and $\lambda\in(0,1)$, with $|x|\ge \nu t^{\theta}$ and $s\in (0,t)$, then
$$
|x-\lambda y|>|x|/2,\qquad |x-\lambda y|(t-s)^{-\theta}>(1-\delta)|x|t^{-\theta}> \nu/2.
$$
Therefore, using the estimate~\eqref{eq:bound.DY} for the gradient of $Y$ and the size assumption~\eqref{eq-hypothesis f},
\[
\begin{aligned}
 {\rm I}(x,t)&\le C\int_0^t\int_{|y|<\delta|x|}|f(y,s)|(t-s)^{2\alpha-1}|x-ry|^{-(N+2\beta+1)}|y|\,{\rm d}y{\rm d}s\\
 &\le
 C\delta E_{-2\beta}(x)\int_0^t\int_{|y|<\delta|x|}|f(y,s)|(t-s)^{2\alpha-1}\,{\rm d}y{\rm d}s\\
 &\le C\delta E_{-2\beta}(x)\int_0^t(1+s)^{-\gamma}(t-s)^{2\a-1}\,{\rm d}s
\end{aligned}
\]	
Thus, since
\begin{equation}
\label{eq:Lp.E-2beta}
\|E_{-2\beta}\|_{L^p(\{|x|>\nu t^{\theta}\})}=Ct^{-\sigma(p)-2\a+1},
\end{equation}
we get
 \begin{equation}
 \label{eq-bound I}
\|{\rm I}(\cdot,t)\|_{L^p(\{|x|>\nu t^{\theta}\})}\le C\delta  t^{-\sigma(p)-2\a+1}\Big(t^{2\a-1}\int_0^{t/2}(1+s)^{-\gamma}\,{\rm d}s+t^{-\gamma}\int_{t/2}^t(t-s)^{2\a-1}\,{\rm d}s\Big),
\end{equation}
which combined with~\eqref{eq:est.integral.t} yields $\|{\rm I}(\cdot,t)\|_{L^p(\{|x|>\nu t^{\theta}\})}\le C\delta\phi(t)$. From now on we fix $\delta\in(0,1/2)$ so that $C\delta <\ep$.

We now turn our attention to ${\rm II}$. If $p\in[1,p_{\rm c})$, using~\eqref{eq:p.norm.Y},
$$
\|{\rm II}(\cdot,t)\|_{L^p(\{|x|>\nu t^{\theta}\})}\le C\int_0^t\|f(\cdot,s)\|_{L^1(\{|x|>\delta\nu t^\theta\})}(t-s)^{-\sigma(p)}\,{\rm d}s.
$$
If moreover $\gamma>1$,  then $f\in L^1(\mathbb{R}^N\times(0,\infty))$, and hence
$\displaystyle\int_0^{t/2}\|f(\cdot,s)\|_{L^1(\{|x|>\delta\nu t^\theta\})}\,{\rm d}s=o(1)$, so that, using also assumption~\eqref{eq-hypothesis f} to estimate the integral over $(t/2,t)$,
$$
\begin{aligned}
\|{\rm II}(\cdot,t)\|_{L^p(\{|x|>\nu t^{\theta}\})}&
\le Ct^{-\sigma(p)}\int_0^{t/2}\|f(\cdot,s)\|_{L^1(\{|x|>\delta\nu t^\theta\})}\,{\rm d}s+Ct^{-\gamma}\int_{t/2}^t(t-s)^{-\sigma(p)}\,{\rm d}s\\&=o(t^{-\sigma(p)})+O(t^{1-\gamma-\sigma(p)})=o(t^{-\sigma(p)})=o(\phi(t)).
\end{aligned}
$$

Still in the subcritical case, if $\gamma\le 1$, using the tail control assumption~\eqref{eq:uniform.tail.control.subcritical} we have
$$
\begin{aligned}
\|{\rm II}(\cdot,t)\|_{L^p(\{|x|>\nu t^{\theta}\})}&
\le C\sup_{t>0}\big( (1+t)^\gamma\|f(\cdot,t)\|_{L^1(\{|x|>\delta\nu t^\theta\})}\big) \int_0^t (1+s)^{-\gamma} (t-s)^{-\sigma(p)}\,{\rm d}s\\
&=o\Big(t^{-\sigma(p)}\int_0^{t/2}(1+s)^{-\gamma}\,{\rm d}s+Ct^{-\gamma}\int_{t/2}^t(t-s)^{-\sigma(p)}\,{\rm d}s\Big)
\\
&=\begin{cases}
o(t^{1-\gamma-\sigma(p)}),&\gamma<1,\\
o(t^{-\sigma(p)}\log t)&\gamma=1,
\end{cases}
\end{aligned}
$$
and therefore $\|{\rm II}(\cdot,t)\|_{L^p(\{|x|>\nu t^{\theta}\})}=o(\phi(t))$.

If $p$ is not subcritical, we take $q$ and $r$ as in~\eqref{eq:choice.q}. Then, since $r$ is subcritical, using~\eqref{eq:p.norm.Y},
\begin{equation}
\label{eq:estimate.exterior.II}
\begin{aligned}
\|{\rm II}(\cdot,t)\|_{L^p(\{|x|>\nu t^{\theta}\})}&\le C\int_0^t\|f(\cdot,s)\|_{L^q(\{|x|>\delta\nu t^\theta\})}(t-s)^{-\sigma(r)}\,{\rm d}s.
\\
&\le C v(t) \Big(t^{-\sigma(r)}\int_0^{t/2}(1+s)^{-\gamma}\,{\rm d}s+t^{-\gamma}\int_{t/2}^t(t-s)^{-\sigma(r)}\,{\rm d}s\Big),
\end{aligned}
\end{equation}
where
\begin{equation}
\label{eq:definition.v}
v(t)=\sup_{s>0}\Big((1+s)^\gamma\|f(\cdot,s)\|_{L^q(\{|x|>\delta\nu t^\theta\})}\Big)=o(t^{-N\theta\big(1-\frac1q\big)}),
\end{equation}
thanks to the uniform tail control assumption~\eqref{eq:uniform.tail.control.not.subcritical}. Therefore
$$
\|{\rm II}(\cdot,t)\|_{L^p(\{|x|>\nu t^{\theta}\})}= \begin{cases}
o(t^{1-\gamma-\sigma(p)}),&\gamma<1,\\
o(t^{-\sigma(p)}\log t),&\gamma=1,\\
o(t^{-\sigma(p)},&\gamma>1,
\end{cases}
$$
whence $\|{\rm II}(\cdot,t)\|_{L^p(\{|x|>\nu t^{\theta}\})}=o(\phi(t))$.
\end{proof}

\begin{proof}[Proof of Theorem~\ref{thm:outer.gamma.ge.1}.]
Take $\delta\in(0,1/2)$. Using hypothesis~\eqref{eq-hypothesis f} on the size of $f$, we have
\[
\begin{aligned}
\Big|\int_0^t M_f(s)&Y(x,t-s)\, {\rm d}s-M(t)t^{1-\alpha}Y(x,t)\Big|\le {\rm I}(x,t)+{\rm II}(x,t)+{\rm III}(x,t),\quad\text{where}\\
{\rm I}(x,t)&=\int_0^{\delta t}(1+s)^{-\gamma}(t-s)^{\alpha-1}|(t-s)^{1-\alpha}Y(x,t-s)-t^{1-\alpha}Y(x,t)|\,{\rm d}s,\\
{\rm II}(x,t)&=\int_{\delta t}^t(1+s)^{-\gamma} Y(x,t-s)\,{\rm d}s,\\
{\rm III}(x,t)&= t^{1-\alpha}Y(x,t)\int_{\delta t}^t(1+s)^{-\gamma}(t-s)^{\alpha-1}\,{\rm d}s,
\end{aligned}
\]
for some $\delta\in(0,1/2)$ to be fixed later.

By the Mean Value Theorem, for each $x$, $t$ and $s\in(0,t)$ there exists $\lambda\in(0,1)$ such that
$$
|(t-s)^{1-\alpha}Y(x,t-s)-t^{1-\alpha}Y(x,t)|=s |\partial_t H(x,t-\lambda s)|, \quad\text{where } H(x,t)=t^{1-\alpha}Y(x,t).
$$
From estimates~\eqref{eq:exterior.estimate.Y} and~\eqref{eq:bound.Yt}, if $|x|t^{-\theta}\ge \nu $, $t>0$, for some $\nu>0$, then
$$
|\partial_t H(x,t)| \le C_\nu t^{\alpha-1}E_{-2\beta}(x).
$$
But, if $|x|>\nu t^\theta$, with $\nu>0$, $s\in(0,\delta t)$, with $\delta\in(0,1/2)$, and $\lambda\in(0,1)$, then
$$
t-\lambda s>t/2,\quad |x|(t-\lambda s)^{-\theta}\ge|x|t^{-\theta}\ge \nu.
$$
Therefore, we have
\[
|(t-s)^{1-\alpha}Y(x,t-s)-t^{1-\alpha}Y(x,t)|\le  C s (t-\lambda s)^{\alpha-1}E_{-2\beta}(x)\le C \delta t^\alpha E_{-2\beta}(x),\]
so that
\[
{\rm I}(x,t)\le  C\delta t^{\a}E_{-2\beta}(x)\int_0^{\delta t}(t-s)^{\alpha-1}(1+s)^{-\gamma}\,{\rm d}s
 \le C\delta  t^{2\alpha-1}E_{-2\beta}(x)\int_0^{\delta t}(1+s)^{-\gamma}\,{\rm d}s.
\]
Using~\eqref{eq:Lp.E-2beta}, we finally get $\|{\rm I}(\cdot,t)\|_{L^p(\{|x|>\nu t^\theta\})}\le \ep\phi(t)$ if we choose $\delta\in(0,1/2)$ small enough.

Once the value of $\delta$ is fixed, we have, using the exterior bound~\eqref{eq:exterior.estimate.Y} for the kernel,
\[
\begin{aligned}
{\rm II}(x,t)&\le CE_{-2\beta}(x)\int_{\delta t}^t(1+s)^{-\gamma}(t-s)^{2\a-1}\,{\rm d}s\le C t^{2\a-\gamma}E_{-2\beta}(x),\\
{\rm III}(x,t)&\le Ct^{2\a-1}E_{-2\beta}(x)\int_{\delta t}^t(1+s)^{-\gamma}\,{\rm d}s\le Ct^{2\a-\gamma}E_{-2\beta}(x),
\end{aligned}\]
so that $\|{\rm II}(\cdot,t)\|_{L^p(\{|x|>\nu t^\theta\})},\|{\rm III}(\cdot,t)\|_{L^p(\{|x|>\nu t^\theta\})}\le C t^{1-\gamma-\sigma(p)}=o(\phi(t))$ if $\gamma\ge1$.
\end{proof}

\begin{proof}[Proof of Theorem~\ref{thm:outer.gamma.gt.1}.]
Let $\delta\in(0,1/2)$  to be chosen later. We have
\[
 \begin{aligned}|M(t)t^{1-\alpha}-M_\infty|&\le \Big(t^{1-\a}\int_0^{\delta t}\int_{\mathbb{R}^N} |f(y,s)|\big((t-s)^{\a-1}-t^{\a-1}\big)\,{\rm d}y{\rm d}s\\
 	 &\qquad +t^{1-\a}
 	 \int_{\delta t}^t\int_{\mathbb{R}^N}|f(y,s)|(t-s)^{\a-1}\,{\rm d}y{\rm d}s+\int_{\delta t}^\infty\int_{\mathbb{R}^N} |f(y,s)|\,{\rm d}y{\rm d}s\Big).
\end{aligned}
\]
Since, by the Mean Value Theorem, $0\le (t-s)^{\a-1}-t^{\a-1}\le C\delta t^{\alpha-1}$ if $s\in (0,\delta t)$, using also the size condition~\eqref{eq-hypothesis f} on $f$ with $\gamma>1$ we conclude that
$$
\begin{aligned}
|M(t)t^{1-\alpha}-M_\infty|&\le C\delta \int_0^{\delta t}(1+s)^{-\gamma}\,{\rm d}s+Ct^{1-\a-\gamma}\int_{\delta t}^t(t-s)^{\a-1}\,{\rm d}s+\int_{\delta t}^\infty\int_{\mathbb{R}^N} |f(y,s)|\,{\rm d}y{\rm d}s\\
&\le C\delta+Ct^{1-\gamma}+\int_{\delta t}^\infty\int_{\mathbb{R}^N} |f(y,s)|\,{\rm d}y{\rm d}s\le \varepsilon,
\end{aligned}
$$
if we fix $\delta$ small enough and then take $t$  large.
\end{proof}

\begin{proof}[Proof of Theorem~\ref{thm:outer.gamma.equal.1}]
We make the estimate $|M(t)-M_0t^{\alpha-1}\log(1+t)|\le {\rm I}(t)+{\rm II}(t)+{\rm III}(t)$,
where
\begin{align*}
{\rm I}(t)&=\int_0^t(t-s)^{\alpha-1}(1+s)^{-1}|(1+s)M_f(s)-M_0|\,{\rm d}s,\\
{\rm II}(t)&=|M_0|\int_0^t\big|(t-s)^{\alpha-1}-t^{\alpha-1}\big|(1+s)^{-1}\,{\rm d}s,\\
{\rm III}(t)&=|M_0|t^{\alpha-1}\log\frac{1+t}t.
\end{align*}

From assumption~\eqref{eq:hypothesis.asymp.sep.variables.1} we know that there is a time $\tau_\varepsilon$ such that
\begin{equation}
\label{eq:approaching.M0}
|\big(1+s)M_f(s)-M_0|\le \|(1+s)f(\cdot,s)-g\|_{L^1(\mathbb{R}^N)}<\ep\quad\text{for all }s\ge\tau_\varepsilon.
\end{equation}
With this in mind, we make the estimate ${\rm I}(t)\le {\rm I}_1(t)+{\rm I}_2(t)$, where
$$
\begin{aligned}
{\rm I}_1(t)&=\int_0^{\tau_\varepsilon}(t-s)^{\alpha-1}(1+s)^{-1}|(1+s)M_f(s)-M_0|\,{\rm d}s,\\
{\rm I}_2(t)&=\varepsilon\int_{\tau_\varepsilon}^{t}(t-s)^{\alpha-1}(1+s)^{-1}\,{\rm d}s,
\end{aligned}
$$
valid for $t>\tau_\varepsilon$. On the one hand,  the size assumption~\eqref{eq-hypothesis f} yields $(1+s)|M_f(s)|\le C$, so that
\[
\begin{aligned}{\rm I}_1(t)&\le C\int_0^{\tau_\varepsilon}(t-s)^{\alpha-1}(1+s)^{-1}\,{\rm d}s
\le C_\varepsilon t^{\alpha-1}\int_0^{\tau_\ep}(1+s)^{-1}\,{\rm d}s\\
&=C_\varepsilon t^{\alpha-1}\log (1+\tau_\varepsilon)
\le \ep t^{\alpha-1}\log (1+t)
\end{aligned}\]
if $t$ is large enough. On the other hand, from~\eqref{eq:approaching.M0}, for all large $t$,
\[
\begin{aligned}
{\rm I}_2(t)&\le C\ep \Big(t^{\alpha-1}\int_{\tau_\ep}^{t/2}(1+s)^{-1}\,{\rm d}s+t^{-1}\int_{t/2}^t(t-s)^{\alpha-1}\,{\rm d}s\Big)
\\
&\le C\ep (t^{\alpha-1}\log (1+t)+ t^{\alpha-1})\le C\varepsilon t^{\alpha-1}\log t.
\end{aligned}
\]

As for ${\rm II}$, we estimate it as ${\rm II}(t)\le {\rm II}_1(t)+{\rm II}_2(t)$, where
\begin{align*}
{\rm II}_1(t)&=|M_0|\int_0^{\delta t}\big|(t-s)^{\alpha-1}-t^{\alpha-1}\big|(1+s)^{-1}\,{\rm d}s,\\
{\rm II}_2(t)&=|M_0|\int_{\delta t}^t\big|(t-s)^{\alpha-1}-t^{\alpha-1}\big|(1+s)^{-1}\,{\rm d}s,\\
\end{align*}
for some $\delta\in(0,1/2)$ to be chosen.
Given $\varepsilon>0$, there exists a small constant $\delta=\delta(\varepsilon)>0$  such that
\[
|(t-s)^{\alpha-1}-t^{\alpha-1}|<\ep t^{\alpha-1}\quad\text{if }s\in(0,\delta t).
\]
We fix such $\delta$. Then, if $t$ is large enough,
$$
{\rm II}_1(t)\le |M_0|\varepsilon t^{\alpha-1}\int_0^{\delta t}(1+s)^{-1}\,{\rm d}s=|M_0|\varepsilon t^{\alpha-1}\log(1+\delta t) \le C\varepsilon t^{\alpha-1}\log t.
$$
On the other hand, for $t$ large enough,
$$
{\rm II}_2(t)\le C t^{-1} \int_{\delta t}^t\big((t-s)^{\alpha-1}+t^{\alpha-1}\big)\,{\rm d}s\le C t^{\alpha-1}\le \varepsilon t^{\alpha-1}\log t.
$$

Finally, since $\log\frac{1+t}t=o(1)=o(\log t)$ as $t\to\infty$, we get immediately that ${\rm III}(t)=o\big(t^{\alpha-1}\log t)$.
\end{proof}

\begin{proof}[Proof of Proposition~\ref{prop-coherence}]
Let $\delta\in(0,1/2)$ to be fixed later. We have
$$
\begin{aligned}
\Big|\int_0^t M_f(s)&Y(x,t-s)\,{\rm d}s-t^{-\gamma}M_0 c_{2\beta}E_{2\beta}(x)\Big|\le {\rm I}(x,t)+{\rm II}(x,t),\quad\text{where }\\
		{\rm I}(x,t)&=\Big|\int_0^{\delta t}\int_{\mathbb{R}^N} f(y,t-s)Y(x,s)\,{\rm d}y{\rm d}s-t^{-\gamma}M_0 c_{2\beta}E_{2\beta}(x)\Big|,\\
        {\rm II}(x,t)&=\int_{\delta t}^t\int_{\mathbb{R}^N} |f(y,t-s)|Y(x,s)\,{\rm d}y{\rm d}s.
\end{aligned}
$$
We estimate ${\rm I}$ as ${\rm I}\le {\rm I}_1+{\rm I}_2$, where
$$
\begin{aligned}
{\rm I}_1(x,t)&=\int_0^{\delta t}(1+t-s)^{-\gamma}Y(x,s)\int_{\mathbb{R}^N} |f(y,t-s)(1+t-s)^{\gamma}-g(y)|\,{\rm d}y{\rm d}s,\\
{\rm I}_2(x,t)&=\Big|M_0\int_0^{\delta t}(1+t-s)^{-\gamma}Y(x,s)\,{\rm d}s-t^{-\gamma}M_0 c_{2\beta}E_{2\beta}(x)\Big|.
\end{aligned}
$$
Let $\varepsilon>0$. Since $t-s\ge t/2$ for $s\in (0,\delta t)$, using hypothesis~\eqref{eq:hypothesis.asymp.sep.variables.1} we get
$$
\int_{\mathbb{R}^N} |f(y,t-s)(1+t-s)^{\gamma}-g(y)|\,{\rm d}y\le \varepsilon |M_0|\quad\text{for }s\in (0,\delta t)
$$
for $t$ large enough, how big not depending on $\delta$, so that
$$
{\rm I}_1(x,t)\le \varepsilon t^{-\gamma}|M_0|E_{2\beta}(x)\int_0^{\delta t}\frac{Y(x,s)}{E_{2\beta}(x)}\,{\rm d}s.
$$
We recall now that
\begin{equation}
\label{eq:recall.A.eq.c2beta}
c_{2\beta}=\int_0^\infty\frac{Y(y,s)}{E_{2\beta}(y)}\,{\rm d}s.
\end{equation}
Therefore, ${\rm I}_1(x,t)\le \varepsilon t^{-\gamma}|M_0|c_{2\beta}E_{2\beta}(x)$.

Using again~\eqref{eq:recall.A.eq.c2beta}, we have ${\rm I}_2\le{\rm I}_{21}+{\rm I}_{22}$, where
$$
\begin{aligned}
{\rm I}_{21}(x,t)&=|M_0|E_{2\beta}(x)\int_0^{\delta t}|(1+t-s)^{-\gamma}-t^{-\gamma}|\,\frac{Y(x,s)}{E_{2\beta}(x)}\,{\rm d}s,\\
{\rm I}_{22}(x,t)&=|M_0|t^{-\gamma}\int_{\delta t}^\infty Y(x,s)\,{\rm d}s.
\end{aligned}
$$
If $s\in(0,\delta t)$, then
\[1-\delta<\frac{1+(1-\delta)t}t<\frac{1+t-s}t<\frac{1+t}t=1+\frac 1 t,
\]
so that $\big|(1+t-s)^{-\gamma}-t^{-\gamma}\big|\le \ep t^{-\gamma}$ if $t$ is large and $\delta $ small. Thus, using once more~\eqref{eq:recall.A.eq.c2beta},
$$
{\rm I}_{21}(x,t)\le \ep t^{-\gamma} |M_0| E_{2\beta}(x) \int_0^{\delta t}\frac{Y(y,s)}{E_{2\beta}(y)}\,{\rm d}s \le  \ep t^{-\gamma} |M_0| c_{2\beta} E_{2\beta}(x).
$$

Once we fix $\delta$ as above, using the global estimate~\eqref{eq:global.estimate.Y} for the kernel,
if $|\xi|=|x|t^{-\theta}$ is small enough,
\[
{\rm I}_{22}(x,t)\le Ct^{-\gamma}E_{4\beta}(x)\int_{\delta  t}^\infty s^{-(1+\alpha)}\,{\rm d} s\le C_\delta(|x|t^{-\theta})^{2\beta} |M_0| c_{2\beta}t^{-\gamma}E_{2\beta} (x)
\le \ep t^{-\gamma} |M_0| c_{2\beta} E_{2\beta}(x).
\]
Similarly, using also the size hypothesis~\eqref{eq-hypothesis f}, if $|\xi|=|x|t^{-\theta}$ is small enough,
\[
\begin{aligned}
{\rm II}(x,t)&\le CE_{4\beta}(x)\int_{\delta t}^t (1+t-s)^{-\gamma}s^{-(1+\a)}\,{\rm d}s
\le C_\delta t^{-(1+\a)}E_{4\beta}(x)\int_{\delta t}^t (1+t-s)^{-\gamma}\,{\rm d}s\\
&\le C_\delta(|x|t^{-\theta})^{2\beta} |M_0| c_{2\beta}t^{-\gamma}E_{2\beta} (x)
\le \ep t^{-\gamma} |M_0| c_{2\beta} E_{2\beta}(x).
\end{aligned}
\]
\end{proof}	

%%%%%%%%%%%%%%%%%%%%%%%%%%%%%%%%%%%%%%%%%%%%%%%%%%%%%%%%%%%%%%%%%%%%%%%%%%%%%%%%%%%%%%%%%%%%%%%
\section*{Appendix}\label{sect-appendix}
\setcounter{equation}{0}
\newcommand{\oy}{\overline Y}
\renewcommand{\d}{\,\mathrm{d}}
\renewcommand{\theequation}{A.\arabic{equation}}
\renewcommand{\thesection}{A}

We study here the behavior at infinity of Riesz potentials, using only integral assumptions, a result of independent interest.
\begin{teo}
\label{thm:behavior.Riesz.potential}
Let $\mu\in (0,N)$. Let $g\in L^1(\R^N)$ and $\displaystyle M=\int_{\mathbb{R}^N}g$. If $p\ge p_\mu:=N/(N-\mu)$ we assume in addition the tail control condition
$$
\begin{gathered}
\|g\|_{L^q(\{|x|>R\})}=O\big(R^{-N(1-\frac1q)}\big)\quad \text{as }R\to\infty\text{ for some }q\in(q_\mu(p),p],
\\
q_\mu(p):=\begin{cases}\displaystyle
\frac{Np}{\mu p + N},&p\in[1,\infty),\\[8pt]
\displaystyle\frac{N}{\mu},&p=\infty.
\end{cases}
\end{gathered}
$$
Let $E_\mu$ and $I_\mu$ as in~\eqref{eq:definition.Riesz.potential}.
Then, if $0<\nu\le \mu<\infty$, for any $p\in[1,\infty]$ we have
$$
R^{N\big(1-\frac 1p\big)-\mu}\|I_{\mu}[g]- ME_\mu\|_{L^p(\{\nu<|x|/R<\mu\})}\to0\quad\mbox{as }R\to\infty.
$$
\end{teo}

\begin{proof} We may assume without loss of generality that $g\neq0$. We have $|I_{\mu}[g]-ME_\mu|\le\textrm{I}+\textrm{II}+\textrm{III}+\textrm{IV}$, where
	$$
	\begin{aligned}
	\textrm{I}(x)&=E_{\mu}(x)\int_{|y|<\gamma|x|}\Big|\frac{|x|^{N-\mu}}{|x-y|^{N-\mu}}-1\Big| |g(y)|\,{\rm d}y,\\
	\textrm{II}(x)&=E_{\mu}(x)\int_{|y|>\gamma|x|}|g(y)|\,{\rm d}y,\\
	\textrm{III}(x)&=\int_{
		\scriptsize\begin{array}{c}|y|>\gamma |x|\\ |x-y|<\delta|x|\end{array}}\frac{|g(y)|}{|x-y|^{N-\mu}}\,{\rm d}y,\\
	\textrm{IV}(x)&=\int_{
		\scriptsize\begin{array}{c}|y|>\gamma |x|\\|x-y|>\delta|x|\end{array}}\frac{|g(y)|}{|x-y|^{N-\mu}}\,{\rm d}y,
	\end{aligned}
	$$
	with $\gamma,\delta>0$  to be chosen later.
	
	On the one hand, if $|y|<\gamma|x|$, with $\gamma\in(0,1)$,
	\[
	\frac1{(1+\gamma)^{N-\mu}}\le\frac{|x|^{N-\mu}}{(|x|+|y|)^{N-\mu}}\le \frac{|x|^{N-\mu}}{|x-y|^{N-\mu}}\le\frac{|x|^{N-\mu}}{(|x|-|y|)^{N-\mu}}\le \frac1{(1-\gamma)^{N-\mu}}.
	\]
	Hence, $\Big|\frac{|x|^{N-\mu}}{|x-y|^{N-\mu}}-1\Big|<\varepsilon/\|g\|_{L^1(\mathbb{R}^N)}$ if $\gamma$ is small enough, and therefore $\textrm{I}(x)\le \varepsilon E_{\mu}(x)$, whence
$$
\|{\rm I}\|_{L^p(\{\nu<|x|/R<\mu\})}\le \varepsilon \|E_\mu\|_{L^p(\{\nu<|x|/R<\mu\})}\le \varepsilon R^{-N\big(1-\frac 1p\big)+\mu}
$$
for all values of $R$.
From now on $\gamma$ is assumed to be fixed.
	
Since $g\in L^1(\mathbb{R}^N)$, $\|g\|_{L^1(\{|x|\ge \nu R\})}\le \varepsilon$ if $R$ is large enough. Hence,
$$
\|{\rm II}\|_{L^p(\{\nu<|x|/R<\mu\})}\le \varepsilon \|E_\mu\|_{L^p(\{\nu<|x|/R<\mu\})}\le \varepsilon R^{-N\big(1-\frac 1p\big)+\mu}
$$
if $R$ is large enough.

To estimate ${\rm III}$, we choose
$$
q=1\text{ if }p\in[1,p_\mu),\quad q\in (q_\mu(p),p] \text{ as in the hypothesis if }p\ge p_\mu,\qquad
1+\frac1p=\frac1q+\frac1r.
$$
Notice that $r\in [1,p_\mu)$ in all cases. Then, using the integrability of $g$ if $p\in[1,p_\mu)$, or the tail control condition otherwise,
\[
\|{\rm III}\|_{L^p(\{\nu<|x|/R<\mu\})}\le \|g\|_{L^q(\{|x|>\gamma\nu R\})}\|E_\mu\|_{L^r(\{|x|<\delta\mu R\})}\le C_{\nu,\mu}\gamma^{-N(1-\frac1q)}\delta^{-N(1-\frac1r)+\mu} R^{-N\big(1-\frac 1p\big)+\mu}.
\]
Since $r\in [1,p_\mu)$, then $-N(1-\frac1r)+\mu>0$. Therefore, taking $\delta>0$ small enough,
$$
\|{\rm II}\|_{L^p(\{\nu<|x|/R<\mu\})}\le\varepsilon  R^{-N\big(1-\frac 1p\big)+\mu}.
$$
	
Finally, once $\gamma$ and $\delta$ are fixed, since
\[
\textrm{IV}(x)\le \delta^{\mu-N}E_\mu(x)\int_{|y|>\gamma|x|}|g(y)|\,{\rm d}y,
\]
we have, using the integrability of $g$,
\[
\|{\rm IV}\|_{L^p(\{\nu<|x|/R<\mu\})}\le C_\delta \|g\|_{L^1(\{|x|>\gamma\nu R\})}\|E_\mu\|_{L^p(\{\nu<|x|/R<\delta\})}\le
\varepsilon  R^{-N\big(1-\frac 1p\big)+\mu},
\]
if $R$ is large enough.
\end{proof}

\noindent\emph{Remark. } The tail control assumption in Theorem~\ref{thm:behavior.Riesz.potential} is satisfied, for instance, if $|g(x)|\le |x|^{-N}$.

\section*{Acknowledgments}

%\subsection*{Funding}

\noindent C. Cort\'azar supported by  FONDECYT grant 1190102 (Chile). 

\noindent F. Quir\'os supported by grants CEX2019-000904-S, PID2020-116949GB-I00, and RED2018-102650-T, all of them funded by MCIN/AEI/10.13039/501100011033, and by the Madrid Government (Comunidad de Madrid – Spain) under the multiannual Agreement with UAM in the line for the Excellence of the University Research Staff in the context of the V PRICIT (Regional Programme of Research and Technological Innovation).

\noindent N. Wolanski supported by European Union's Horizon 2020 research and innovation programme under the Marie Sklodowska-Curie grant agreement No.\,777822,  CONICET PIP 11220150100032CO 2016-2019; UBACYT 20020150100154BA, ANPCyT PICT2016-1022 and MathAmSud 13MATH03 (Argentina).


\begin{thebibliography}{99}
	%	\bibitem{Allen-Caffarelli-Vasseur-2016} Allen, M.; Caffarelli, L.; Vasseur, A. \emph{A parabolic problem with a fractional time derivative.} Arch. Ration. Mech. Anal. 221 (2016), no.\,2, 603--630.

\bibitem{Biler-Guedda-Karch-2004} Biler, P.; Guedda, M.; Karch, G. \emph{Asymptotic properties of solutions of the viscous Hamilton-Jacobi equation.} J. Evol. Equ. 4 (2004), no.~1, 75–-97.
	
\bibitem{Caputo-1967} Caputo, M. \emph{Linear models of dissipation whose $Q$ is almost frequency independent--II.} Geophys.
	J. R. Astr. Soc. 13 (1967), 529--539.
	
%\bibitem{Caputo-1999} Caputo, M. \emph{Diffusion of fluids in porous media with memory.} Geothermics, 28 (1999), 113--130.

\bibitem{Cartea-delCastilloNegrete-2007} Cartea, \'A.; del Castillo-Negrete, D. \emph{Fluid limit of the continuous-time random walk with general Lévy jump distribution functions.} Phys. Rev. E 76 (2007) 041105.

\bibitem{Compte-Caceres-1998} Compte, A.; C\'aceres, M.\,O. \emph{Fractional dynamics in random velocity fields.} Phys. Rev. Lett. 81 (1998) 3140--3143.

\bibitem{Cortazar-Quiros-Wolanski-2021a} Cortazar, C.; Quir\'os, F.; Wolanski, N. \emph{A heat equation with memory: large-time behavior.} J. Funct. Anal. 281 (2021), no.\,9, 109174.

\bibitem{Cortazar-Quiros-Wolanski-2021b} Cort\'azar, C.; Quir\'os, F.; Wolanski, N. \emph{Large-time behavior for a fully nonlocal heat equation.} Vietnam J. Math. 49 (2021), no.\,3, 831--844.

\bibitem{Cortazar-Quiros-Wolanski-2022} Cortazar, C.; Quir\'os, F.; Wolanski, N.     \emph{Decay/growth rates for inhomogeneous heat equations with memory. The case of large dimensions.}   Math. Eng. 4 (2022), no.\,3, 1--17.

\bibitem{Cortazar-Quiros-Wolanski-2021-Preprint} Cortazar, C.; Quir\'os, F.; Wolanski, N.     \emph{Decay/growth rates for inhomogeneous heat equations with memory. The case of small dimensions.} Preprint, \texttt{arXiv:2204.11342 [math.AP]}.


\bibitem{Cortazar-Quiros-Wolanski-2022-Preprint} Cort\'azar, C.; Quir\'os, F.; Wolanski, N. \emph{Asymptotic profiles for inhomogeneous classical and fractional heat equations}. Preprint.

\bibitem{delCastilloNegrete-Carreras-Lynch-2004} del Castillo-Negrete, D.; Carreras, B.\,A.; Lynch, V.\,E. \emph{Fractional diffusion in plasma turbulence.}
Physics of Plasmas 11 (2004), no.\,8,  3854--3864.

\bibitem{delCastilloNegrete-Carreras-Lynch-2005} del Castillo-Negrete, D.; Carreras, B.\,A.; Lynch, V.\,E. \emph{Nondiffusive transport in plasma turbulence:
A fractional diffusion approach.} Physical Review Letters 94 (2005), no.\,6,  065003.

\bibitem{Dolbeault-Karch-2006} Dolbeault, J.; Karch, G. \emph{Large time behaviour of solutions to nonhomogeneous diffusion equations.} In \lq\lq Self-similar solutions of nonlinear PDE'', 133–147, Banach Center Publ., 74, Polish Acad. Sci. Inst. Math., Warsaw, 2006.

\bibitem{Dzherbashyan-Nersesian-1968} Dzherbashyan, M.\,M.; Nersesian, A.\,B. \emph{Fractional derivatives and the Cauchy problem for
differential equations of fractional order}.  (Russian) Izv. Akad. Nauk Arm. SSR, Mat. 3 (1968), 3–-29.

	
	\bibitem{Eidelman-Kochubei-2004}
	Eidelman, S.\,D.; Kochubei, A.\,N. \emph{Cauchy problem for fractional diffusion equations.} J. Differential Equations 199 (2004), no.\,2, 211--255.
	
%	\bibitem{Grafakos} Grafakos, L. Classical and Modern Fourier Analysis. Pearson Education, London (2004)


\bibitem{Gerasimov-1948}
Gerasimov, A.\,N. \emph{A generalization of linear laws of deformation and its application to problems of internal friction}. (Russian) Akad. Nauk SSSR. Prikl. Mat. Meh. 12 (1948), 251--260.



%\bibitem{Gripenberg-1985} Gripenberg, G. \emph{Volterra integro-differential equations with accretive nonlinearity.} J. Differential Equations 60 (1985), no.\,1, 57--79.


\bibitem{Gross-1947} Gross, B. \emph{On creep and relaxation}. J. Appl. Phys. 18 (1947), 212--221.

%\bibitem{Hardy-Wright-1979} Hardy, G.\,H.; Wright, E.\,M. \lq\lq An Introduction to the Theory of Numbers'', 5th ed. Oxford, England: Clarendon Press, 1979.


%	
%	\bibitem{Herraiz-1999} Herraiz, L. \emph{Asymptotic behaviour of solutions of some semilinear parabolic problems.} Ann. Inst. H. Poincar\'e Anal. Non Line\'aire 16 (1999), no.\,1, 49--105.
	
%\bibitem{Kemppainen-Siljander-Vergara-Zacher-2016} Kemppainen, J.; Siljander, J.; Vergara, V.; Zacher, R. \emph{Decay estimates for time-fractional and other non-local in time subdiffusion equations in Rd.} Math. Ann. 366 (2016), no.\,3-4, 941--979.
	
	
	\bibitem{Kemppainen-Siljander-Zacher-2017} Kemppainen, J.; Siljander, J.; Zacher, R. \emph{Representation of solutions and large-time behavior for fully nonlocal diffusion equations.} J. Differential Equations 263 (2017), no.\,1, 149--201.
	
\bibitem{Kim-Lim-2016} Kim, K.-H.; Lim, S. \emph{Asymptotic behaviors of fundamental solution and its derivatives to fractional diffusion-wave equations}. J. Korean Math. Soc. 53 (2016), no.\,4, 929--967.

\bibitem{Kochubei-1990} Kochubeĭ, A.\,N. \emph{Diffusion of fractional order}. (Russian) Differentsial'nye Uravneniya 26 (1990), no.\,4, 660--670, 733--734; translation in Differential Equations 26 (1990), no.\,4, 485--492.	

%\bibitem{Landkof} Landkof, N.~S.
%\lq\lq Foundations of modern potential theory'',
%Die Grundlehren der mathematischen Wissenschaften, Band 180. Springer-Verlag,
%New York-Heidelberg, 1972.


\bibitem{Liouville-1832} Liouville, J. \emph{Memoire sur quelques questions de g\'eometrie et de m\'eecanique, et sur un nouveau
gentre pour resoudre ces questions}. J. Ecole Polytech. 13 (1832), 1--69.


	
%	\bibitem{Meerschaert-Sikorskii-2012} Meerschaert, M.\,M.; Sikorskii, A.  Stochastic models for fractional calculus''.
%	De Gruyter Studies in Mathematics, 43. Walter de Gruyter \& Co., Berlin, 2012. ISBN: 978-3-11-025869-1.
	
	\bibitem{Metzler-Klafter-2000}	Metzler, R.; Klafter, J. \emph{The random walk's guide to anomalous diffusion: a fractional dynamics approach.} Phys. Rep. 339 (2000), no.\,1, 77 pp.

%\bibitem{Pruss-book} Pr\"uss, J. \lq\lq Evolutionary integral equations and applications''. [2012] reprint of the 1993 edition. Modern Birkh\"auser Classics. Birkh\"auser/Springer Basel AG, Basel, 1993.  ISBN: 978-3-0348-0498-1.

\bibitem{Rabotnov-1966}    Rabotnov, Yu.\,N. \lq\lq Polzuchest Elementov Konstruktsii''. (Russian) Nauka, Moscow (1966); English translation:
\lq\lq Creep Problems in Structural Members''. North-Holland, Amsterdam (1969).


%	\bibitem{Schneider-Wyss-1989} Schneider, W.\,R.; Wyss, W.
%	\emph{Fractional diffusion and wave equations.}
%	J. Math. Phys. 30 (1989), no.\,1, 134--144.
	
%	\bibitem{Shlesinger-Klafter-Wong-1982} Shlesinger, M.\,F.; Klafter, J.; Wong, Y.\,M. \emph{Random walks with infinite spatial and temporal moments.}
%	J. Statist. Phys. 27 (1982), no.\,3, 499--512.

\bibitem{Stein-book-1970} Stein, E. M. \lq\lq Singular integrals and differentiability properties of functions''. Princeton Mathematical Series, No. 30 Princeton University Press, Princeton, N.J. 1970.
	
%\bibitem{Zacher-2009} Zacher, R. \emph{Weak solutions of abstract evolutionary integro-differential equations in Hilbert spaces.} Funkcial. Ekvac. 52 (2009), no.\,1, 1--18.

\bibitem{Zaslavsky-2002} Zaslavsky, G.\,M. \emph{Chaos, fractional kinetics, and anomalous transport.} Phys. Rep. 371 (2002), no.\,6, 461--580.
	
\end{thebibliography}
\end{document}